\documentclass[11pt,reqno]{amsart}
\usepackage{bm}
\usepackage{a4wide}
\usepackage{mathtools}
\usepackage{undertilde}
\usepackage{mathabx}						% meilleurs \wedge
\usepackage{stmaryrd}                       % Pour les intervalles d'entiers
\usepackage{marginnote}

\newtheorem{thm}{Theorem}
\newtheorem{prop}{Proposition}
\newtheorem{lem}{Lemma}
\newtheorem{coro}{Corollary}
\newtheorem{pb}{Problem}
\newtheorem{remark}{Remark}

\newtheorem{assumption}{{\bf Assumption}}

\newcommand\R{\mathbb{R}}
\newcommand\Z{\mathbb{Z}}

\newcommand\Sph{\mathbb{S}}
\newcommand\osig{\widetilde{\sigma}}

\newcommand*{\transp}[2][-3mu]{\ensuremath{\mskip1mu\prescript{\smash{ t\mkern#1}}{}{\mathstrut#2}}} %transposee de matrice

\begin{document}

\large

\title[Positive Definite Quadratic Form \& Lattice Points]{On the Minimum of a Positive Definite Quadratic Form over Non--Zero Lattice points. Theory and Applications.}
\author{Faustin Adiceam}
\author{Evgeniy Zorin}
\address{FA, EZ: Department of Mathematics, University of York,  York, YO10
5DD, UK} \email{ faustin.adiceam@york.ac.uk, evgeniy.zorin@york.ac.uk,}
%\date{\today}
\thanks{FA research is supported by EPSRC Programme Grant~: EP/J018260/1 and EZ research is supported by EPSRC Grant~: EP/M021858/1.}

\begin{abstract}
Let $\Sigma_d^{++}$ be the set of positive definite matrices with determinant 1 in dimension $d\ge 2$. Identifying any two $SL_d(\Z)$--congruent elements in $\Sigma_d^{++}$  gives rise to the space of reduced quadratic forms of determinant one, which in turn can be identified with  the locally symmetric space $X_d:=SL_d(\Z)\backslash SL_d(\R)\slash SO_d(\R)$. Equip the latter space with its natural probability measure coming from a Haar measure on $SL_d(\R)$. In 1998, Kleinbock and Margulis~\cite{zbMATH01383858} established sharp estimates for the probability that an element of $X_d$ takes a value less than a given real number $\delta>0$ over the non--zero lattice points $\Z^d\backslash\{ \bm{0} \}$. 

In this article, these estimates are extended to a large class of probability measures arising either from the spectral or the Cholesky decomposition of an element of $\Sigma_d^{++}$. The sharpness of the bounds thus obtained are also established (up to  multiplicative constants) for a subclass of these measures.

Although of an independent interest, this theory is partly developed here with a view towards application to Information Theory. More precisely, after providing a concise introduction to this topic fitted to our needs, we lay the theoretical foundations of the study of some manifolds frequently appearing in the theory of Signal Processing. This is then applied to the recently introduced Integer--Forcing Receiver Architecture channel whose importance stems from its expected high performance. Here, we give sharp estimates for the probabilistic distribution of the so--called \emph{Effective Signal--to--Noise Ratio}, which is an essential quantity in the evaluation of the performance of this model.
\end{abstract}

\maketitle

\begin{center}
\emph{In honorem Henriettae Dickinsonis.}
\end{center}

\setcounter{tocdepth}{2}
\tableofcontents

\section{Introduction}

Fix once and for all an integer $d\ge 2$. Let $Q$ be a non--degenerate symmetric matrix in dimension $d$. Throughout, the matrix $Q$ will be identified with the corresponding quadratic form $\bm{x}\in\R^d \mapsto \transp{\bm{x}}\cdot Q \cdot\bm{x}.$

If $Q$ is indefinite, the Oppenheim conjecture solved by Margulis states that the set of values taken by this quadratic form at non--zero integral points, viz. $$\left\{\transp{\bm{a}}\cdot Q \cdot\bm{a}\; : \; \bm{a}\in\Z^d\backslash\{\bm{0}\} \right\},$$ is dense in the real line whenever $d\ge 3$. When $d=2$ however (i.e.~for indefinite binary quadratic forms), this set may exhibit very different structures~: it may be dense or else closed and discrete, but it may also be not closed and/or not dense. For further details on the theory of values taken by an indefinite quadratic form, the reader is referred to~\cite{courtois, dalbo} and to the references therein.

In the case that $Q$ is definite, say positive definite without loss of generality, it is easy to see that the quantity 
\begin{equation}\label{defM_d(Q)}
M_d(Q)\, := \, \min_{\bm{a}\in\Z^d\backslash\left\{\bm{0} \right\}} \transp{\bm{a}}\cdot Q \cdot\bm{a}
\end{equation} 
is well--defined. It is a result due to Hermite (see~\cite[p.43]{zbMATH03141378} for a proof) that one has always 
\begin{equation}\label{defmQ}
M_d(Q)\, \le \, \left(\frac{4}{3} \right)^{(d-1)/2} \left| Q\right|^{1/d},
\end{equation} 
where $\left| Q\right|$ denotes the determinant of $Q$. It is known that the constant $(4/3)^{(d-1)/2}$ on the right--hand side of~\eqref{defmQ} is optimal only when $d=2$. Denoting by $\mathcal{S}_d^{++}$ the set of positive definite matrices in dimension $d\ge 2$, this leads one to the definition of the \emph{Hermite constant} $\gamma_d$~: $$\gamma_d\, := \, \frac{\sup_{Q\in\mathcal{S}_d^{++}}\; M_d(Q)}{\left| Q\right|^{1/d}}\cdotp$$ The supremum in this definition can actually be replaced with a maximum. Only the values of $\gamma_d$ for $d= 2, 3,4, 5, 6, 7, 8$ and $d=24$ are exactly known. For other $d$'s, several estimates have been established. See, e.g., \cite{conway1993sphere} for proofs and further details on the Hermite constants. See also~\cite{zbMATH00435565} for an algorithm to approximate $M_d(Q)$ for a \emph{given} $Q\in \mathcal{S}_d^{++}$.

\paragraph{}It should be noted that the study of the quantity $M_d(Q)$ for a generic $Q\in \mathcal{S}_d^{++}$ underpins the more general problem of determining the minimum of such a quadratic form over non--zero elements of \emph{any} full rank lattice $\Lambda$. Indeed, as such a lattice can be written in the form $\Lambda = L\cdot\Z^d$ for some $L\in GL_d(\R)$, the minimum of $Q$ over the elements of $\Lambda\backslash\{\bm{0} \}$ is given by $M_d\left(\transp{L} Q L\right)$. Also, if $L'\in GL_d(\R)$ is another matrix such that $\Lambda = L'\cdot\Z^d$, then there exists $Z\in SL_d(\Z)$ such that $L'=L Z$. This implies in particular that $M_d(Q) = M_d\left(\transp{Z}QZ\right) $ for any $Q\in \mathcal{S}_d^{++}$  and any $Z\in SL_d(\Z)$ , i.e.~that the quantity $M_d(Q)$ is invariant under $SL_d(\Z)$--congruent matrices. 

\paragraph{} The problem of estimating $M_d(Q)$ is here considered from a probabilistic point of view. Given an estimate such as~\eqref{defmQ}, even if it means renormalising in an obvious way the matrices under consideration, it is natural to focus on the case of positive definite matrices \emph{with determinant one}. Let therefore $$\Sigma_d^{++} \, :=\, \left\{\Sigma\in \mathcal{S}_d^{++} \; : \; \det(\Sigma)=1  \right\}$$ denote such a set. In full generality, the main problem addressed in this work can loosely be summarised this way~: 

\begin{pb}[Main Problem]\label{pb1}
For a \emph{given} probability measure $\mu$ on the set $\Sigma_d^{++}$, estimate the probability $\mu\left(M_d(\Sigma) \le \delta\right)$ as a function of $\delta>0$.
\end{pb}

In order to take into account the $SL_d(\Z)$--invariance of the problem, identify any two $SL_d(\Z)$--congruent matrices in $\Sigma_d^{++}$. This defines the space of reduced quadratic forms with determinant one, which is henceforth denoted by $\Sigma_{d, red}^{++}$. It is easy to see that the map 
\begin{equation}\label{defphi}
\phi\; : \; \overline{g}\in X_d \, \mapsto \, g\cdot \transp{g} \in \Sigma_{d, red}^{++}
\end{equation} 
is well--defined and bijective, where $X_d$ denotes the locally symmetric space $$X_d\,:=\, SL_d(\Z)\backslash SL_d(\R)\slash SO_d(\R)$$ and where $\overline{g} := SL_d(\Z)\cdot g\cdot SO_d(\R)$ is the equivalence class in $X_d$ of any $g\in SL_d(\R)$ (the surjectivity of the map $\phi$ follows for instance from the Cholesky decomposition of an element of $\Sigma_d^{++}$). From now on, let $$\Gamma := SL_d(\Z), \quad G:=SL_d(\R)\quad \textrm{ and } \quad H:=SO_d(\R)$$ (which are all unimodular groups) in such a way that $X_d:=\Gamma\backslash G / H.$

The set $X_d$ seen as a double coset space can be equipped with a natural $G$--invariant probability measure $\mu_{X_d}$ arising from the $G$--invariant probability measure $\mu_{\Gamma\backslash G}$ on the space of lattices $\Gamma\backslash G$. If one denotes by $\mu_H$ the Haar probability measure on $H$, the invariant measure $\mu_{X_d}$ is characterised by the fact that for any Borel measurable function $f\in\mathbb{L}^1(\mu_{\Gamma\backslash G})$, the following equation holds~: $$\int_{X_d}\left(\int_{H} f(gh)\cdot\textrm{d}\mu_H(h) \right)\cdot \textrm{d}\mu_{X_d}(gH)\, = \, \int_{\Gamma\backslash G} f(g)\cdot\textrm{d}\mu_{\Gamma\backslash G}(g)$$ (see~\cite{tsliu} for proofs and details). The probability measure $\mu_{\Gamma\backslash G}$ is itself obtained from any suitably normalised Haar measure $\mu_G$ on $G$. One can furthermore explicitly express the volume element $\textrm{d}\mu_{G} (M)$ in terms of the Iwasawa decomposition of $M\in G$ --- see~\cite[\S 2]{zbMATH05991430} for details.

With the help of the bijective map~\eqref{defphi}, the measure $\mu_{X_d}$ can be pushed forward to a probability measure $\phi_\ast \mu_{X_d}$ on the space $\Sigma_{d, red}^{++}$. In view of Problem~\ref{pb1}, one is then concerned with the estimate of the probability 
\begin{align*}
p_{X_d}(\delta)\, & = \, \left(\phi_\ast \mu_{X_d} \right) \left(\left\{\overline{\Sigma}\in \Sigma_{d, red}^{++}\; : \; M_d(\overline{\Sigma})\le \delta \right\}\right) \\ 
&= \, \mu_{X_d} \left(\left\{\overline{g}\in X_d\; : \; M_d(\phi(\overline{g}))\le \delta \right\}\right)
\end{align*}
for any fixed $\delta>0$ which may be assumed to be less than the Hermite constant $\gamma_d$ for obvious reasons (note that the above equations are direct consequences of the change of variables formula for pushforward measures). This problem was emphatically solved by Kleinbock--Margulis who proved in~\cite[\S 7]{zbMATH01383858} the following result (see also~\cite[Theorem 1.3.5]{klshst}). Before stating it, and in view of the statement of our own results, let from now 
\begin{equation}\label{defVdAd}
V_d\, = \, \frac{\pi^{d/2}}{\Gamma\left(\frac{d}{2}+1\right)} \qquad \textrm{ and }\qquad A_d\,= \, \frac{2\pi^{d/2}}{\Gamma\left(\frac{d}{2}\right)}
\end{equation} 
denote respectively the volume and the area of the unit Euclidean ball in dimension $d\ge 2$ (here, $\Gamma(\, . \,)$ denotes the usual Euler Gamma function). 

\begin{thm}[Kleinbock \& Margulis, 1998]\label{thmkleimarg}
The following inequalities hold for any $\delta>0$~: 
\begin{equation}\label{ineklmar}
\frac{V_d}{2 \zeta(d)} \delta^{d/2} - c_d \frac{V_d^2}{4}\delta^{d} \; \le \; p_{X_d}(\delta)\; \le \; \frac{V_d}{2 \zeta(d)} \delta^{d/2} \cdotp
\end{equation} 
Here, $\zeta$ denotes the Riemann zeta function and $c_d$ a strictly positive constant which, when $d\ge 3$, can be taken to be $$c_d\; = \; \frac{1}{\zeta(d)\cdot \zeta(d-1)}\cdotp$$
\end{thm}

The implicit presence of the square root of $\delta$ on both sides of~\eqref{ineklmar} is due to this easily verified equivalence valid for any $g\in G$~: $$ \left( M_d(\phi(\overline{g}))\,\le\, \delta \right) \quad \iff \quad \left(g\cdot\Z^d\cap B_2(\bm{0}, \sqrt{\delta}) \, \neq\, \left\{\bm{0} \right\}\right),$$ where, given $\bm{x}\in\R^d$ and $r>0$, $B_2(\bm{x}, r)$ is the closed Euclidean ball with radius $r$ centered at $\bm{x}$.

Theorem~\ref{thmkleimarg} suggests that, as $\delta>0$ tends to zero, one should expect the probability of the event $M_d(\Sigma)\le \delta$ to grow like $\delta^{d/2}$ when the space $\Sigma_{d}^{++}$ is equipped with a ``typical''  probability measure defined from the invariant measure $\mu_{X_d}$. For the applications we have in mind however (see~\S\ref{secsignalproc}), the choice of any such measure is neither natural nor convenient. The primary theoretical goal of this work is thus to establish estimates in the likes of~\eqref{ineklmar} for a larger class of probability measures on the space $\Sigma_{d}^{++}$. These probability measures will be defined from the spectral (\S\ref{approachspectral}) and then the Cholesky decomposition (\S\ref{approachcholesky}) of an element of $\Sigma_{d}^{++}$. 

Note that, although the problem of estimating the probability of the event $M_d(\Sigma)\le \delta$ is well--defined in the space $\Sigma_{d, red}^{++}$ of reduced quadratic forms, there is no loss of information in working instead in the space $\Sigma_{d}^{++}$. Indeed, any probability measure on $\Sigma_{d}^{++}$ defines a probability measure on $\Sigma_{d, red}^{++}$ after periodisation modulo $SL_d(\Z)$--congruent matrices. Conversely, any probability measure on $\Sigma_{d, red}^{++}$  defines a probability measure on $\Sigma_{d}^{++}$ supported on a fundamental domain of $\Sigma_{d, red}^{++}$ in $\Sigma_{d}^{++}$.

Before stating the main results, we mention that the latter may also be used to tackle the following less natural but nevertheless still relevant variant of the main problem stated above (namely, when the probability space is $\mathcal{S}_d^{++}$ instead of $\Sigma_d^{++}$)~:

\begin{pb}[Variant of the Main Problem]\label{pb2}
For a \emph{given} probability measure $\mu'$ on the set $\mathcal{S}_d^{++}$,  estimate the probability $\mu'\left(M_d(Q) \le \delta\right)$ as a function of $\delta>0$.
\end{pb}

The changes to make to the results dealing with Problem~\ref{pb1} in order to obtain their analogues for Problem~\ref{pb2} are straightforward when considering the approach via the spectral decomposition (\S\ref{approachspectral}). They will therefore not be explicitly stated. When considering the approach via the Cholesky decomposition however (\S\ref{approachcholesky}), these changes will induce some technical difficulties  and will therefore be explicitly stated. 

Throughout, in order not interrupt the thread of the exposition, the lengthy proofs are postponed until the end of each section. They may be skipped at a first reading.

\section{An Approach via the Spectral Decomposition.} \label{approachspectral}

Denote by $\mathcal{D}_d^{++}$ the set of diagonal matrices in dimension $d$ with strictly positive entries. Let $\Delta_d^{++}$ be the subgroup of $\mathcal{D}_d^{++}$ consisting of all those matrices with determinant one~: $$\Delta_d^{++} := \mathcal{D}_d^{++}\cap SL_d(\R).$$ Throughout, $\mathcal{D}_d^{++}$  (resp.~$\Delta_d^{++}$) will be identified   with $(\R_{>0})^d$ (resp.~with $(\R_{>0})^{d-1}$ --- in this case, one only considers the $d-1$ first diagonal entries of an element of $\Delta_d^{++}$ to define the identification). It will sometimes be more convenient to see an element of $\Delta_d^{++}$ as an element of $\mathcal{D}_d^{++}$, in which case it will also be represented as a $d$--tuple. This should not cause any confusion.

Let $$\mathcal{O}_d:=O_d(\R)$$ denote the orthogonal group in dimension $d$. We first seek to equip the set $\Sigma_{d}^{++}$ with a special class of probability measures defined from the spectral decomposition of an element therein. This class will play an important role in the forthcoming considerations~: in short, Problem~\ref{pb1} will be addressed for probability measures lying in this class. 

\subsection{Definition of a Suitable Class of Measures} 

Let $\Sigma\in \Sigma_{d}^{++}$ be decomposed as $\Sigma = \transp{P}\Delta P$ with $P\in\mathcal{O}_d$ and $\Delta\in \Delta_d^{++}$. Given $\bm{x}\in\R^d$, one has clearly $\transp{\bm{x}}\cdot \Sigma \cdot\bm{x} \, = \, \transp{\bm{y}}\cdot \bm{y}$ with $\bm{y}=\sqrt{\Delta}P\bm{x}$. This shows that the following equivalence holds for any $\delta>0$~: 
\begin{equation}\label{equivspect}\left(M_d(\Sigma)\, \le \, \delta\right) \quad \iff \quad \left( P\cdot\Z^d \: \cap\: \Delta^{-1/2}\cdot B_2(\bm{0}, \sqrt{\delta})\, \neq\, \left\{\bm{0}\right\} \right).
\end{equation}
This motivates the introduction of the surjective map 
\begin{equation}\label{defpsi}
\Psi\; : \; (P, \Delta)\in \mathcal{O}_d\times \Delta_d^{++} \; \mapsto \; \transp{P}\Delta^{-2} P \in\Sigma_{d}^{++} 
\end{equation} 
which we now use to push forward to $\Sigma_{d}^{++}$ a given measure defined on $\mathcal{O}_d\times \Delta_d^{++}$ (the exponent ``-2'' is just meant to simplify the formulae hereafter). It is important to keep in mind for what follows that the orthogonal matrix $P$ appearing in the Spectral Decomposition of $\Sigma$ as above is well--defined in the quotient $\mathcal{O}_d\slash \mathcal{I}_d$, where $\mathcal{I}_d$ is the subgroup of $\mathcal{O}_d$ consisting of all those diagonal matrices with entries $\pm 1$. The equivalence~\eqref{equivspect} then still holds when $P$ is seen as an element of $\mathcal{O}_d\slash \mathcal{I}_d$ in view of the fact that $P\cdot I\cdot \Z^d=P\cdot \Z^d$ for any $I\in\mathcal{I}_d$.

Let $\mu_d$ be the Haar probability measure on the compact group $\mathcal{O}_d$. Given $P\in\mathcal{O}_d$, the volume element $\textrm{d}\mu_d(P)$ is explicitly described for instance in~\cite{zbMATH03380964} in terms of $d(d-1)/2$ independent coordinates on $\mathcal{O}_d$. Let furthermore $\nu_d$ be a probability measure on $\Delta_d^{++}$. Define then a measure on the product space $\mathcal{O}_d\times \Delta_d^{++}$ by setting
\begin{equation}\label{deftaud}
\tau_d\, := \, \mu_d \otimes\nu_d.
\end{equation}  
This can be pushed forward to a probability measure $\Psi_{\ast}\tau_d$ on $\Sigma_{d}^{++}$. Of course, the relevance of such a measure strongly relies on the properties of the map $\Psi$ and of the measure $\tau_d$. In this respect, the following lemma establishes a crucial property satisfied by $\Psi$~:

\begin{lem}\label{impromptu}
Let $\Delta_{d, sub}^{++}$ be the subset of $\Delta_d^{++}$ consisting of all those elements in $\Delta_d^{++}$ whose entries are pairewise distinct~: $$\Delta_{d, sub}^{++}\; := \; \left\{\Delta=(\alpha_1, \cdots, \alpha_d)\in \Delta_d^{++} \; : \; \forall i\neq j, \; \alpha_i\neq\alpha_j \right\}.$$
Then, the restriction of the map $\Psi$ to the set $\mathcal{O}_d\times \Delta_{d, sub}^{++}$ is $2^d$ to 1. 

More precisely, $\Psi$ induces a bijection
\begin{equation}\label{defpsiprime} 
\Psi'\; : \; \left(\mathcal{O}_d/\mathcal{I}_d\right)\times \Delta_{d, sub}^{++}\; \mapsto \; \Psi\left(\mathcal{O}_d\times \Delta_{d, sub}^{++}\right)\subset\Sigma_{d}^{++}. 
\end{equation}
\end{lem}

Note that $\Psi\left(\mathcal{O}_d\times \Delta_{d, sub}^{++}\right)$ sits as a dense open set in $\Sigma_{d}^{++}$.

\begin{proof}
Let $Q\in\Sigma^{++}_d$ with spectral decomposition $Q=\transp{P}\Delta^{-2}P$ for some $P\in\mathcal{O}_d$ and some $\Delta\in\Delta^{++}_{d, sub}$. The rows of the matrix $P$ are then (normed) eigenvectors of $Q$. Since eigenvectors associated to distinct eigenvalues are  orthogonal, these rows are determined up to their sign. The lemma follows.
\end{proof}

Let $\rho_d$ be the Haar probability measure on $\mathcal{O}_d/\mathcal{I}_d$, which satisfies the property that for any function $f\in\mathbb{L}^1(\mu_d)$ defined over $\mathcal{O}_d$, \begin{equation}\label{measrho}
\int_{\mathcal{O}_d} f(P)\cdot \textrm{d}\mu_d(P)\; = \; \frac{1}{2^d}\cdotp \int_{\mathcal{O}_d/\mathcal{I}_d} \left(\sum_{I\in \mathcal{I}_d} f(PI)\right)\cdot \textrm{d}\rho_d(P\mathcal{I}_d).
\end{equation}

In view of Lemma~\ref{impromptu}, a dense open subset of $\Sigma_{d}^{++}$ can be identified with the product space $\left(\mathcal{O}_d/\mathcal{I}_d\right)\times \Delta_{d, sub}^{++}$ via the map $\Psi'$ defined in~\eqref{defpsiprime}. We will be interested in probability measures supported on this dense open set. A natural class of such measures are obtained by taking the pushforward by $\Psi'$  of a measure of the form $\rho_d\otimes \nu_d$ under the following assumption on $\nu_d$ which will be made throughout~:

\begin{assumption}\label{bighyp}
The complement of  $\Delta_{d, sub}^{++}$  in  $\Delta_{d}^{++}$ has zero $\nu_d$--measure, i.e. $$\nu_d\left(\Delta_{d, sub}^{++} \right)=1.$$
\end{assumption}

Thus, under this assumption, $\Psi'$ establishes a bijection between a set of full $\rho_d\otimes \nu_d$--measure in $\left(\mathcal{O}_d/\mathcal{I}_d\right)\times \Delta_{d}^{++}$ and its image in  $\Sigma_{d}^{++}$.

Note also that under Assumption~\ref{bighyp},  the two pushforward measures $\Psi'_\ast(\rho_d\otimes \nu_d)$ and $\Psi_\ast \tau_d$ (with $\tau_d$ defined in~\eqref{deftaud}) are exactly the same on $\Sigma_{d}^{++}$. Indeed, if $\Sigma\in\Sigma_{d}^{++}$  lies in the image of the restriction of the map $\Psi$  to $\mathcal{O}_d\times \Delta_{d, sub}^{++}$, Lemma~\ref{impromptu} implies that the preimage $\Psi^{-1}\left(\left\{\Sigma\right\} \right)$ of $\Sigma$ by $\Psi$ is of the form $\Psi^{-1}\left(\left\{\Sigma\right\} \right) = \left\{ (PI, \Delta)\; : \; I\in\mathcal{I}_d\right\}$ for some $P\in\mathcal{O}_d$ and $\Delta\in \Delta_d^{++}$. Since the orthogonal matrix $P$ appearing in the the equivalence stated in~\eqref{equivspect} can be seen as an element of $\mathcal{O}_d/\mathcal{I}_d$, it follows from the definition of $\Psi$ in~\eqref{defpsi} that either all or none of the $2^d$ elements $(P, \Delta)$ in this preimage satisfy/ies the relation 
\begin{equation}\label{equivtronquee}
P\cdot\Z^d \: \cap\: \Delta\cdot B_2(\bm{0}, \sqrt{\delta})\, \neq\, \left\{\bm{0}\right\}.
\end{equation}

Together with~\eqref{measrho}, this establishes the claim.

Assumption~\ref{bighyp} imposes a rather mild restriction on the measure $\nu_d$, which is even allowed to be fractal. A natural class of measures satisfying this assumption is given by those probability measures which are absolutely continuous with respect to a Haar measure $\xi$ on $\Delta_{d}^{++}$. Recall that, up to a multiplication constant, the volume element $\textrm{d}\xi(\Delta)$ of any such invariant measure is given by 
\begin{equation}\label{haardiagsld}
\textrm{d}\xi(\Delta)\, = \, \prod_{i=1}^{d-1}\frac{\textrm{d}\alpha'_i}{\alpha'_i},
\end{equation} 
where $\Delta=(\alpha'_1, \dots, \alpha'_{d-1})\in \Delta_{d}^{++}$.

\subsection{Estimation of the Probability that a Non--Zero Integer Vector should lie in a Random Ellipsoid Centered at the Origin.} We adopt here a geometric approach in order to address Problem~\ref{pb1} within the framework developed thus far. Part of the ideas behind this approach have been applied in~\cite{zbMATH05991430} to problems in mathematical physics. However, unlike here, the focus in the latter work was rather on the probability that a \emph{large} convex set should contain a non--zero lattice point. Furthermore, the multiplicative constants appearing in the formulae proved in~\cite{zbMATH05991430} are not explicit while it will be one of our objectives to obtain fully explicit estimates. 

From the change of variables formula for pushforward measures and in view of~\eqref{equivspect}, \eqref{defpsi} and~\eqref{equivtronquee}, the objective boils down to estimating, for a given $\delta>0$, the quantity 
\begin{align*}
\left(\Psi_\ast\tau_d\right)&\left(\left\{ \Sigma\in\Sigma_{d}^{++}\; : \; M_d(\Sigma)\, \le \, \delta\right\} \right) \; = \; \tau_d\left(\mathfrak{F}_{d}(\delta)\right),
\end{align*}
where 
\begin{equation*}\label{defF_d}
\mathfrak{F}_{d}(\delta)\; := \; \left\{\left(P, \Delta\right)\in \mathcal{O}_d\times\Delta_{d}^{++}\; : \; P\cdot\Z^d \: \cap\: \Delta\cdot B_2(\bm{0}, \sqrt{\delta})\, \neq\, \left\{\bm{0}\right\} \right\}.
\end{equation*}
To avoid cumbersome notation, the set $\mathfrak{F}_{d}(\delta)$ will from now on be denoted by $\mathfrak{F}(\delta)$ whenever there is no risk of confusion.

In order to state the results regarding the estimate of the probability $\tau_d\left(\mathfrak{F}(\delta)\right)$, a good deal of notation is first introduced.

Throughout, a vector in $\R^d$ will be seen as the datum of a $d$--tuple represented in \emph{column} (that is, we consider the right action of $d$--dimensional matrices on $\R^d$). Whenever this does not induce any ambiguity, such a vector shall indifferently be written in row for convenience. Given a vector $\bm{\alpha}:=(\alpha_1, \dots, \alpha_d)\in (\R_{>0})^d$, $\mathcal{E}_d\left(\bm{\alpha}\right)$ will denote the \emph{full} ellipsoid 
\begin{equation}\label{defellipsoide}
\mathcal{E}_d\left(\bm{\alpha}\right)\, :=\, \left\{\bm{x}\in\R^d\; : \; \sum_{i=1}^d \left(\frac{x_i}{\alpha_i} \right)^2\, \le \, 1 \right\}
\end{equation} 
($\alpha_1,\dots, \alpha_d$ are thus the lengths of the semi--principal axes of this ellipsoid). If there is no risk of confusion, one shall also write more simply $\mathcal{E}(\bm{\alpha})$ for $\mathcal{E}_d\left(\bm{\alpha}\right)$. 

Let $\Sph^{d-1}$ denote the unit sphere in dimension $d$. Let also $\sigma_{d-1}$ be the spherical probability measure on $\Sph^{d-1}$. This measure is given by a volume element denoted by $\textrm{d}\bm{v}$ which is such that for any $\sigma_{d-1}$--measurable surface $\mathcal{A}\subset \Sph^{d-1}$, $$\sigma_{d-1}\left(\mathcal{A}\right)\, := \, \frac{1}{A_d}\int_{\mathcal{A}}\textrm{d}\bm{v}$$ (we have chosen not to include the factor $A_d$ in the volume element as otherwise any use of our results will unavoidably involve the computation of constants involving this factor). If $\mathcal{A}$ is any subset of $\R^d$ such that its intersection $\mathcal{A}\cap \Sph^{d-1}$ with the unit sphere is $\sigma_{d-1}$--measurable, set $$\osig_{d-1}\left(\mathcal{A}\right)\; :=\; \sigma_{d-1}\left( \mathcal{A}\cap \Sph^{d-1}\right).$$ Given a vector $\bm{v}\in\Sph^{d-1}$, $\bm{v}^{\perp}$ shall denote the hyperplane in $\R^d$ passing through the origin with unit normal vector $\bm{v}$. Also, the notation $\left\|\, . \, \right\|_2$ and $\left\|\, . \, \right\|_\infty$ shall refer to the usual Euclidean and sup norms in $\R^d$. The set of points in $\Z^d$ visible from the origin shall be denoted by $\mathcal{P}(\Z^d)$~: $$\mathcal{P}(\Z^d)\, := \, \left\{\bm{a}\in\Z^d\; : \; \gcd(\bm{a})=1 \right\}.$$ Finally, given a closed convex set $\mathcal{C}\subset \R^d$ centered at the origin, define $$p_d(\mathcal{C})\; :=\; \mu_d\left( \left\{P\in\mathcal{O}_d \; : \; P\cdot\Z^d \cap \mathcal{C} \neq \left\{\bm{0} \right\}\right\}\right).$$ Note that in the case $d=1$, $\mathcal{O}_1=\left\{\pm 1\right\}$, the convex body $\mathcal{C}$ is an interval $\mathcal{J}$ and 
\begin{equation}\label{casbase}
p_1(\mathcal{J})\;=\; \left\{
    \begin{array}{ll}
        1 & \mbox{if  }\, \lambda\left( \mathcal{J}\right)\ge 2\\
        0 & \mbox{if  }\, \lambda\left( \mathcal{J}\right)< 2, 
    \end{array}
\right.
\end{equation}
where $\lambda\left( \mathcal{J}\right)$ denotes the length of $\mathcal{J}$.

The main result in this section can now be stated as follows.

\begin{thm}\label{thmprinci}
Let $\delta>0$. Then,
\begin{equation}\label{introresultprinci}
\tau_d\left(\mathfrak{F}(\delta)\right)\; = \; \int_{\Delta_d^{++}} p_d\left(\mathcal{E}(\sqrt{\delta}\Delta)\right)\cdot  \textrm{\emph{d}}\nu_d(\Delta).
\end{equation}  
Furthermore, the quantity $p_d\left(\mathcal{E}(\sqrt{\delta}\Delta)\right)$ satisfies the estimates 
\begin{equation}\label{resultprinci}
g_d(\Delta, \delta)\; \le \; p_d\left(\mathcal{E}(\sqrt{\delta}\Delta)\right) \; \le \; f_d(\Delta, \delta),
\end{equation} 
where
\begin{align*}
g_d(\Delta, \delta)\; := \; \max & \left\{\osig_{d-1} \left( \mathcal{E}_d(\sqrt{\delta}\Delta)\right), \;  \int_{\Sph^{d-1}} p_{d-1}\left(\mathcal{E}_d(\sqrt{\delta}\Delta)\cap\bm{v}^{\perp}\right)\cdot\frac{\emph{\textrm{d}}\bm{v}}{A_d} \right\} 
\end{align*} 
and
\begin{align*}
f_d(\Delta, \delta)\; := \; \min\left\{1, \; \sum_{\underset{\left\|\bm{n}\right\|_2\, \le \, \sqrt{\delta}\left\| \Delta\right\|_\infty}{\bm{n}\in \mathcal{P}(\Z^d)}} \osig_{d-1} \left( \mathcal{E}_d\left(\frac{\sqrt{\delta}}{\left\|\bm{n}\right\|_2}\Delta\right)\right)\right\}.
\end{align*}
Here, the base case for the recursive formula induced by the integral in $g_d(\Delta, \delta)$ is given by~\eqref{casbase} and the sum in $f_d(\Delta, \delta)$ is to be seen as equal to zero when $\sqrt{\delta}\left\| \Delta\right\|_\infty<1$.
\end{thm}

In view of such a statement, we now seek to determine, one the one hand the intersection of an ellipsoid with a hyperplane and on the other the spherical measure of the intersection of a (full) ellipsoid with the unit sphere. The former question is addressed in this proposition~:

\begin{prop}\label{propinterellipse}
Let $\bm{\alpha}=(\alpha_1, \dots, \alpha_d)\in (\R_{>0})^d$ and $\bm{v}=(v_1, \dots, v_d)\in\Sph^{d-1}$. Assume that $v_d\neq 0$.

Then, the intersection $\mathcal{E}_d\left(\bm{\alpha}\right)\cap \bm{v}^{\perp}$ of the $d$--dimensional ellipsoid $\mathcal{E}_d\left(\bm{\alpha}\right)$ with the hyperplane $\bm{v}^{\perp}$ is a $(d-1)$--dimensional ellipsoid $\mathcal{E}_{d-1}\left(\bm{\alpha}, \bm{v}\right)$. Furthermore, one has 
\begin{equation}\label{interellipse}
\mathcal{E}_{d-1}\left(\bm{\alpha}, \bm{v}\right)\; = \; \left\{ \bm{y}\in\R^{d-1}\; : \; \transp{\bm{y}}\cdot Q \cdot \bm{y}\, \le \, 1\right\},
\end{equation} 
where 
\begin{equation}\label{interellipsebis}
Q\, :=\, D\left(I_{d-1}+\bm{u}\cdot\transp{\bm{u}} \right) D\, \in \, \mathcal{S}_d^{++} 
\end{equation} 
with $I_{d-1}$ the identity matrix in dimension $d-1$, $$D\, :=\, \left(\alpha_1^{-1}, \cdots, \alpha_d^{-1} \right)\in \mathcal{D}_d^{++} \quad \mbox{ and } \quad \transp{\bm{u}}\; :=\; \left(\frac{\alpha_i v_i}{\alpha_d v_d} \right)_{1\le i \le d-1}\in \R^{d-1}.$$

Also, if the lengths of the semi--principal axes of $\mathcal{E}_d\left(\bm{\alpha}\right)$ are ordered increasingly in the sense that $\alpha_1\le \dots \le\alpha_d$, then the lengths $\beta_1, \dots, \beta_{d-1}$ of the semi--principal axes of $\mathcal{E}_{d-1}\left(\bm{\alpha}, \bm{v}\right)$ ordered increasingly satisfy the inequalities $$\alpha_1\, \le \, \beta_1\, \le \, \alpha_2\, \le \, \dots\, \le \, \alpha_{d-1}\, \le \, \beta_{d-1}\, \le \, \alpha_d.$$
\end{prop}

Note that, even if it means relabelling the axes, there is no loss of generality in assuming that the lengths of the semi--principal axes of $\mathcal{E}_d\left(\bm{\alpha}\right)$ are ordered increasingly. Also, the condition $v_d\neq 0$ is not restrictive at all as formula~\eqref{interellipse} holds \emph{mutatis mutandis} with any other non--zero coordinate $v_j$ in place of $v_d$ --- see the proof in \S\ref{secellemt} for details.

We now turn to the estimate of the spherical measure of the intersection of the ellipsoid $\mathcal{E}_d\left(\bm{\alpha}\right)$ with the unit sphere (where $\bm{\alpha}=(\alpha_1, \dots, \alpha_d)\in (\R_{>0})^d$). To this end, it may be assumed, without loss of generality in view of Assumption~\ref{bighyp}, that 
\begin{equation}\label{inegaslpha}
0\, <\,\alpha_1\, < \, \alpha_2\, < \, \dots\, < \, \alpha_{d-1}\, < \, \alpha_d.
\end{equation}
Whenever $\alpha_d>1$, define then 
\begin{equation}\label{quiproquo}
\bm{\utilde{\alpha}}\, :=\, \left(\utilde{\alpha}_1, \, \dots\, , \, \utilde{\alpha}_{d-1} \right)\, \in \, \mathcal{D}_{d-1}^{++},
\end{equation} 
where for $i=1, \dots, d-1$, $$\utilde{\alpha}_i\, := \, \sqrt{\alpha_i^2\cdot \frac{\alpha_d^2-1}{\alpha_d^2-\alpha_i^2}}.$$

The following statement provides an inductive formula for $\osig_{d-1} \left( \mathcal{E}_d\left(\bm{\alpha}\right)\right)$. The quantity
\begin{equation}\label{wallis}
W_k\, = \, \int_{0}^{\pi/2}\sin^k\theta\cdot\textrm{d}\theta\, = \, \frac{\sqrt{\pi}}{2}\cdot \frac{\Gamma\left(\frac{k+1}{2} \right)}{\Gamma\left(\frac{k+2}{2} \right)}
\end{equation} 
appearing therein denotes the Wallis integral of order $k\ge 0$.
\begin{prop}\label{propformulerecu}
Assuming~\eqref{inegaslpha}, one has 
\begin{equation}\label{castriviaux}
\osig_{d-1} \left( \mathcal{E}_d\left(\bm{\alpha}\right)\right)\;=\; \left\{
    \begin{array}{ll}
        1 & \mbox{if } \alpha_1\ge 1\\
        0 & \mbox{if } \alpha_d\le 1. 
    \end{array}
\right.
\end{equation}
Moreover, if $\alpha_1<1<\alpha_d$, then 
\begin{equation}\label{formulerecu}
\osig_{d-1} \left( \mathcal{E}_d\left(\bm{\alpha}\right)\right)\;=\; \frac{1}{2 W_{d-2}}\cdot \int_{0}^{\pi} \osig_{d-2} \left( \mathcal{E}_{d-1}\left(\frac{\bm{\utilde{\alpha}}}{\sin \theta}\right)\right)\cdot \left(\sin \theta \right)^{d-2}\cdot \emph{\textrm{d}}\theta 
\end{equation} 
with base case 
\begin{equation*}
\osig_{0} \left( \mathcal{E}_1\left(\alpha\right)\right)\;=\; \left\{
    \begin{array}{ll}
        1 & \mbox{if } \alpha\ge 1\\
        0 & \mbox{if } \alpha< 1 
    \end{array}
\right.
\end{equation*}
for any $\alpha>0$. 
\end{prop}

Although providing an exact theoretical formula, equation~\eqref{formulerecu} may lead to lengthy calculations for a given ellipsoid. In order to overcome this difficulty, the next proposition provides rather accurate estimates for the quantity $\osig_{d-1} \left( \mathcal{E}_d\left(\bm{\alpha}\right)\right)$ when $\alpha_1<1<\alpha_d$. Before stating it, we introduce some additional notation~: given $x\ge 0$, let
\begin{equation*}
b(x)\; :=\; \arccos \left(\min\left\{1, x\right\}\right)\, = \, \left\{
    \begin{array}{ll}
        \arccos(x) \in [0, \, \pi/2] & \mbox{if } x\in[0, \, 1],\\
        0 & \mbox{if } x\ge 1. 
    \end{array}
\right.
\end{equation*}
Under~\eqref{inegaslpha}, define 
\begin{equation}\label{defIdalspha}
\mathfrak{I}_d(\bm{\alpha})\, := \, \frac{2^{d}}{A_d}\cdot \prod_{i=2}^{d}\int_{b(\alpha_{d-i+1})}^{\pi/2} \sin^{i-2}\theta\cdot\textrm{d}\theta \quad \textrm{whenever} \quad \alpha_d\ge 1.
\end{equation}
We leave this quantity undefined when $\alpha_d<1$. For $i=1, \dots, d-1$, assuming $\alpha_d\ge 1$, set furthermore 
\begin{equation*}
\utilde{\alpha}_{i}^*\, := \, \min\left\{1,  \utilde{\alpha}_{i}\right\}  \, = \, 
\left\{
    \begin{array}{ll}
        \utilde{\alpha}_{i} & \mbox{if }\alpha_i\le 1,\\
        1 & \mbox{if } \alpha_i\ge 1 
    \end{array}
\right.
\end{equation*}
and let $\bm{\utilde{\alpha}}^* = \left(\utilde{\alpha}_{1}^*, \, \dots ,\, \utilde{\alpha}_{d-1}^* \right)$.

\begin{prop}\label{propestimineg}
Assume that~\eqref{inegaslpha} holds and that $\alpha_1<1<\alpha_d$. Then, with the notation above, one has 
\begin{equation*}
\mathfrak{I}_d\left(\frac{\bm{\utilde{\alpha}}^*}{\sqrt{d-1}}, 1 \right) \, \le \, \osig_{d-1}\left(\mathcal{E}_d\left(\bm{\alpha} \right) \right)\, \le \, \mathfrak{I}_d\left(\bm{\utilde{\alpha}}^*, 1 \right).
\end{equation*}
The following cruder but easier--to--estimate inequalities also hold~: 
\begin{equation*}
\mathfrak{I}_d\left(\frac{\bm{\alpha}}{\sqrt{d}} \right) \, \le \, \osig_{d-1}\left(\mathcal{E}_d\left(\bm{\alpha} \right) \right)\, \le \, \mathfrak{I}_d\left(\bm{\alpha} \right),
\end{equation*}
where the lower bound is defined whenever $\alpha_d\ge \sqrt{d}$.

Here, given a generic vector $\bm{\alpha}\in (\R_{>0})^d$ satisfying~\eqref{inegaslpha} and $\alpha_d\ge 1$, the quantity $\mathfrak{I}_d(\bm{\alpha})$  can be estimated as follows~: 
\begin{equation*}
a(d)\cdot \prod_{j=1}^{d-1}\min\left\{\alpha_j, \, 1 \right\}\, \le \, \mathfrak{I}_d\left(\bm{\alpha}\right) \, \le \, a'(d)\cdot \prod_{j=1}^{d-1}\min\left\{\alpha_j, \, 1 \right\}
\end{equation*}
with
\begin{equation*}
a(d)\, = \, \frac{2^{d}}{(d-1)! \cdot A_d}\cdot \left( \frac{\pi}{2}\right)^{(d-2)(d-3)/2} \quad \textrm{ and } \quad a'(d)=\frac{2^{d}}{A_d}\cdot \left( \frac{\pi}{2}\right)^{d(d-1)/2}.
\end{equation*}
\end{prop}

With the help of Propositions~\ref{propinterellipse}, \ref{propformulerecu} and~\ref{propestimineg}, one may now answer the question as to whether Theorem~\ref{thmprinci} leads to sharp estimates for the probability $\tau_d\left(\mathfrak{F}(\delta)\right)$ as expressed in~\eqref{introresultprinci}. To this end, one must focus on a relevant subclass of probability measures $\nu_d$. A natural choice is to restrict the attention to compactly supported measures. Indeed, such measures can approximate a large class of measures and appear naturally in practical problems (see~\S\ref{secsignalproc}). Assume therefore without loss of generality that $\nu_d$ seen as a measure on $(\R_{>0})^{d-1}$ is absolutely continuous with respect to the Haar measure~\eqref{haardiagsld} with density supported on the hypercube $[\epsilon, \, \epsilon^{-1}]^{d-1}$. Denote by $\chi_\epsilon^{(d)}~: \R^{d-1}\rightarrow \R$ the characteristic function of the latter set.

To simplify the calculations, we will further require that the density of $\nu_d$ with respect to the Haar measure $\xi$ is uniform, i.e.~that $\xi$--almost everywhere, the density $\textrm{d}\nu_d/\textrm{d}\xi$ is proportional to $\chi_\epsilon^{(d)}$. In view of~\eqref{haardiagsld}, given $\bm{\alpha'} = (\alpha'_1, \dots, \alpha'_{d-1})\in  \Delta_d^{++}$, one has explicitly
\begin{equation}\label{defvdepsi}
\textrm{d}\nu^{(\epsilon)}_d(\bm{\alpha'})\, = \, \frac{1}{\left|2\log \epsilon \right|^{d-1}}\cdot \chi_\epsilon^{(d)}(\bm{\alpha'})\cdot \prod_{i=1}^{d-1}\frac{\textrm{d}\alpha'_i}{\alpha'_i},
\end{equation} 
where $\nu^{(\epsilon)}_d = \nu_d$. Inasmuch as one is working up to multiplicative constants, one can reduce to this case any measure whose density with respect to $\nu^{(\epsilon)}_d $ is almost everywhere bounded above on the hypercube $K_{\varepsilon}(d)=[\epsilon, \, \epsilon^{-1}]^{d-1}$ and almost everywhere bounded below by a strictly positive constant on a sub--hypercube of $K_{\varepsilon}(d)$.

The next proposition shows that, for any given $\epsilon>0$, the estimates of the probability $\tau^{(\epsilon)}_d\left(\mathfrak{F}(\delta)\right)\,:=\,\tau_d\left(\mathfrak{F}(\delta)\right)$ obtained from Theorem~\ref{thmprinci} are essentially sharp in $\delta$.

\begin{thm}\label{thmapplicunif}
Fix $\epsilon>0$ and assume that $\delta\in (0,\, 1)$. Let $\tau^{(\epsilon)}_d$ be the probability measure defined as in~\eqref{deftaud} from the measure $\nu^{(\epsilon)}_d$ given by~\eqref{defvdepsi}.

Then,
\begin{equation}\label{probanulle} 
\tau^{(\epsilon)}_d\left(\mathfrak{F}(\delta)\right)\, = \, 0 \quad \textrm{ if }\quad \delta\le \epsilon^{2(d-1)}.
\end{equation} 
Moreover, if $\delta>\epsilon^{2(d-1)}$, then 
\begin{equation}\label{inegprobaepsi}
c_d(\epsilon)\cdot s_d\left(\epsilon, \delta \right)\, \le \, \tau^{(\epsilon)}_d\left(\mathfrak{F}(\delta)\right)\, \le \, C_d(\epsilon) \cdot S_d\left(\epsilon, \delta \right)
\end{equation} 
for some constants $c_d(\epsilon),  C_d(\epsilon)>0$.  Here, 
\begin{equation*}\label{defs_d}
s_d\left(\epsilon, \delta \right)\, :=\, \int_{J_d(\epsilon, \delta)}\prod_{i=1}^{d-1}\min\left\{\sqrt{\delta}, \, \frac{1}{\alpha_i} \right\}\cdot \emph{\textrm{d}}\alpha_i
\end{equation*}
and
\begin{equation*}\label{defS_d}
S_d\left(\epsilon, \delta \right)\, :=\, \delta^{d/2}\cdot\int_{J_d(\epsilon, \delta)}\prod_{i=1}^{d-1}\frac{\emph{\textrm{d}}\alpha_i}{\alpha_i},
\end{equation*}
where the domain of integration $J_d(\epsilon, \delta)$ is defined by the set of inequalities $$\epsilon \le \alpha_1 <\dots < \alpha_{d-1}\le \epsilon^{-1} \quad \textrm{and} \quad \max\left\{\delta^{-1/2}, \, \alpha_{d-1} \right\}<\left(\alpha_1\dots \alpha_{d-1} \right)^{-1}.$$
These quantities $s_d\left(\epsilon, \delta \right)$ and $S_d\left(\epsilon, \delta \right)$ satisfy the estimates
\begin{align}
s_d\left(\epsilon, \delta \right) \, \ge \, \min\left\{\sqrt{\delta}, \epsilon \right\}^{d-1} \cdot \frac{\left|2\log \epsilon \right|^{d-2}}{(d-2)!}\cdot\left(\min\left\{\sqrt{\delta}, \epsilon \right\}-\epsilon^{d-1} \right). \label{inegs_d1}
\end{align}
and
\begin{align}
S_d\left(\epsilon, \delta \right) \, \le \, \delta^{d/2}\cdot \log\left(\frac{\sqrt{\delta}}{\epsilon^{d-1}} \right)\cdot \frac{\left|2\log \epsilon \right|^{d-2}}{(d-2)!}\cdotp \label{inegs_d2}
\end{align}
One can furthermore choose $$c_d(\epsilon)\, = \, \frac{a(d)\cdot (d-1)!}{\left(d\cdot \left|2\log \epsilon\right| \right)^{d-1}}$$ and $$C_d(\epsilon)\, =\, \frac{3^{d-1}\cdot a'(d)\cdot d!\cdot d}{\left|2\cdot \log \epsilon \right|^{d-1}},$$ where $a(d)$ and $a'(d)$ are defined in Proposition~\ref{propestimineg}.
\end{thm}

Theorem~\ref{thmapplicunif} implies for instance the existence of two positive constants $\kappa(d)$ and $K(d)$ depending \emph{only} on the dimension $d$ such that for any $\delta$ lying in the interval $\left[\epsilon^{2(d-1)}, \, \epsilon^2\right]$, $$\kappa(d)\cdot \frac{\delta^{d/2}}{\left|\log \epsilon \right|}\cdot \left(1-\frac{\epsilon^{d-1}}{\sqrt{\delta}} \right)\, \le \, \tau^{(\epsilon)}_d\left(\mathfrak{F}(\delta)\right)\, \le \, K(d)\cdot \frac{\delta^{d/2}}{\left|\log \epsilon \right|}\cdot \left(\frac{\sqrt{\delta}}{\epsilon^{d-1}}-1\right)$$ (the upper bound is a direct consequence of the convexity inequality $\log (1+x)\le x$ valid for all $x\ge 0$). We thus recover in this case also the growth in $\delta^{d/2}$ appearing in Theorem~\ref{thmkleimarg}.

The remainder of this section is devoted to the proofs of the various results stated above.

\subsection{Proof of Theorem~\ref{thmprinci}}

Note that equation~\eqref{introresultprinci} follows immediately from Fubini's Theorem applied to the probability measure $\tau_d$. The upper and lower bounds in~\eqref{resultprinci} will now be established separately. To this end, we first make the following crucial remark~: if $\mathcal{A}\subset \Sph^{d-1}$ is a $\sigma_{d-1}$--measurable set and $\bm{x_0}\in\Sph^{d-1}$, then 
\begin{align}
\sigma_{d-1}\left(\mathcal{A}\right) \;= \; \mu_d\left(\left\{G\in\mathcal{O}_d\; : \; G\bm{x_0}\in\mathcal{A} \right\} \right). \label{messpehr2}
\end{align}
Indeed, each of the measures involved in this equation is clearly Borelian and uniformly distributed on the unit sphere (in the sense that the measure of a ball on the sphere depends only on the radius of the ball but not on the position of its centre). Now, a result of Christensen~\cite{zbMATH03304410} states that two Borelian measures uniformly distributed in a separable metric space must be proportional. As the measures under consideration have been normalised to become probability measures, they must be equal --- see~\cite[Chap.~3]{zbMATH03141379} for details.

\begin{proof}[Proof of the upper bound in~\eqref{resultprinci}.]
Let $\delta>0$ and $\Delta\in \Delta^{++}_{d}$. 
The symmetry with respect of the origin and the convexity of the ellipsoid $\mathcal{E}_d\left(\sqrt{\delta}\Delta \right)$ imply that 
\begin{align*}
\left\{P \in \mathcal{O}_d\; : \; P\cdot\Z^d \: \cap\: \mathcal{E}_d\left(\sqrt{\delta}\Delta \right) \, \neq\, \{\bm{0}\} \right\}\qquad\qquad\qquad\qquad\qquad\qquad\qquad\qquad\\ 
\qquad \qquad\qquad\qquad= \, \left\{P \in \mathcal{O}_d\; : \; P\cdot\mathcal{P}\left(\Z^d\right) \: \cap\: \mathcal{E}_d\left(\sqrt{\delta}\Delta \right) \, \neq\, \emptyset \right\}. 
\end{align*} 
Given an event $\mathfrak{E}$, let $\chi_{\mathfrak{E}}$ denote the Boolean function 
\begin{equation*}\label{boolean}
\chi_{\left[\mathfrak{E}\right]}\;=\; \left\{
    \begin{array}{ll}
        1 & \mbox{if } \mathfrak{E} \mbox{ holds}\\
        0 & \mbox{if } \mathfrak{E} \mbox{ does not holds}. 
    \end{array}
\right.
\end{equation*}
Then, denoting by $\# S$ the cardinality of a finite set $S$, one has
\begin{align}
p_d\left(\mathcal{E}(\sqrt{\delta}\Delta)\right)\, &=\, \int_{\mathcal{O}_d}\textrm{d}\mu_d(P)\cdot\chi_{\left[P\cdot\mathcal{P}\left(\Z^d\right) \: \cap\: \mathcal{E}_d\left(\sqrt{\delta}\Delta \right) \, \neq\, \emptyset\right]}\label{prsvt}\\
&\le \, \int_{\mathcal{O}_d}\textrm{d}\mu_d(P)\cdot \#\left( P\cdot\mathcal{P}\left(\Z^d\right) \: \cap\: \mathcal{E}_d\left(\sqrt{\delta}\Delta \right)\right)\nonumber\\
&  = \, \int_{\mathcal{O}_d}\textrm{d}\mu_d(P)\cdot \left(\sum_{\bm{n}\in \mathcal{P}(\Z^d)} \chi_{\left[P\bm{n}\:\in\:\mathcal{E}_d\left(\sqrt{\delta}\Delta\right)\right]}\right).\nonumber
\end{align}
Now, given $P\in \mathcal{O}_d$ and $\bm{n}\in\mathcal{P}(\Z^d)$, it should be clear that $$P\bm{n}\, \in \, \mathcal{E}_d\left(\sqrt{\delta}\Delta\right) \qquad \iff \qquad P\frac{\bm{n}}{\left\|\bm{n}\right\|_2}\, \in\, \mathcal{E}_d\left(\frac{\sqrt{\delta}}{\left\|\bm{n}\right\|_2}\cdot\Delta\right)\cap \Sph^{d-1}.$$ For either of these statements to be true, it is furthermore necessary that $$\left\| \bm{n}\right\|_2\, \le \, \sqrt{\delta}\cdot \left\|\Delta\right\|_{\infty}.$$ Therefore, 
\begin{align*}
p_d\left(\mathcal{E}(\sqrt{\delta}\Delta)\right)\, &\le\, \sum_{\underset{\left\|\bm{n}\right\|_2\, \le \, \sqrt{\delta}\left\| \Delta\right\|_\infty}{\bm{n}\in \mathcal{P}(\Z^d)}}\mu_d\left(\left\{P\in\mathcal{O}_d\; : \; P\frac{\bm{n}}{\left\|\bm{n}\right\|_2}\, \in\, \mathcal{E}_d\left(\frac{\sqrt{\delta}}{\left\|\bm{n}\right\|_2}\cdot\Delta\right)\cap \Sph^{d-1} \right\}\right)\\
&\underset{\eqref{messpehr2}}{=} \, \sum_{\underset{\left\|\bm{n}\right\|_2\, \le \, \sqrt{\delta}\left\| \Delta\right\|_\infty}{\bm{n}\in \mathcal{P}(\Z^d)}} \osig_{d-1}\left( \mathcal{E}_d\left(\frac{\sqrt{\delta}}{\left\|\bm{n}\right\|_2}\cdot\Delta\right)\right),
\end{align*}
hence the claim.
\end{proof}
 
\begin{proof}[Proof of the lower bound in~\eqref{resultprinci}.] Let $\bm{e_1}=\transp{(1, 0, \dots, 0)\in\R^d}$ be the first element of the standard vector basis in $\R^d$. It then follows from~\eqref{prsvt} that 
\begin{align*}
p_d\left(\mathcal{E}(\sqrt{\delta}\Delta)\right)\, &\ge\, \mu_d\left(\left\{ P\in\mathcal{O}_d\; : \; P\bm{e_1}\in  \mathcal{E}_d\left(\sqrt{\delta}\Delta\right)\right\} \right)\\
&\underset{\eqref{messpehr2}}{=} \, \osig_{d-1}\left(\mathcal{E}_d\left(\sqrt{\delta}\Delta\right) \right),
\end{align*}
which establishes the first of the two inequalities to be proved. 

The proof of the second one is more involved. Let $\bm{e_d}=\transp{(0, \dots, 0, 1)\in\R^d}$ denote the last element of the standard vector basis in $\R^d$. Letting the group $\mathcal{O}_d$ act on the sphere $\Sph^{d-1}$, the stabiliser of $\bm{e_d}$ is isomorphic to $\mathcal{O}_{d-1}$ identified with the subgroup $$ \begin{pmatrix} \mathcal{O}_{d-1} & \bm{0} \\ \transp{\bm{0}} & 1 \end{pmatrix} \, \subset \, \mathcal{O}_d.$$ With this identification, given $R, S \in\mathcal{O}_d$, the product $S^{-1}R$ lies in $\mathcal{O}_{d-1}$ if, and only if the last columns of $R$ and $S$ are the same, i.e. $$S^{-1}R\in\mathcal{O}_{d-1} \quad \iff \quad R\bm{e_d} = S\bm{e_d}\, \in \, \Sph^{d-1}.$$ This implies the well--known fact that  the quotient $\mathcal{O}_d/\mathcal{O}_{d-1}$ is isomorphic to the sphere $\Sph^{d-1}$. Fix now a measurable function $f~: \Sph^{d-1}\rightarrow\mathcal{O}_d$ such that 
\begin{equation}\label{?1}
\forall \bm{v}\in\Sph^{d-1}, \quad f(\bm{v})\cdot\bm{e_d}\,=\,\bm{v}.
\end{equation}
Any $S\in \mathcal{O}_d$ can then be written uniquely in the form
\begin{equation}\label{?2}
S\, = \, f(\bm{v})\cdot \begin{pmatrix} S' & \bm{0} \\ \transp{\bm{0}} & 1 \end{pmatrix},
\end{equation}
where $S'\in \mathcal{O}_{d-1}$ and $\bm{v}\in\Sph^{d-1}$ (in particular, the last column of $S$ is then $\bm{v}$).

Furthermore, if $R, S\in\mathcal{O}_{d}$ are respectively represented by $(R', \bm{u})$ and $(S', \bm{v})$ in these coordinates (where $R', S' \in\mathcal{O}_{d-1}$ and $\bm{u}, \bm{v}\in\Sph^{d-1}$), then $RS$ is represented by $(T'S', R\bm{v})$ for some $T'\in\mathcal{O}_{d-1}$ depending only on $R$ and $\bm{v}$. Indeed, this follows from the uniqueness of the representation~\eqref{?2} together with~\eqref{?1} which implies that the last column of $R\cdot f(\bm{v})$ is $R\bm{v}$. Thus, identifying $\mathcal{O}_d$ with $\mathcal{O}_{d-1}\times \Sph^{d-1}$, left multiplication on $\mathcal{O}_d$ by some $R\in \mathcal{O}_d$ induces a left multiplication on $\mathcal{O}_{d-1}$ by some $T'\in\mathcal{O}_{d-1}$ (depending only on $R$ and $\bm{v}$) and the orthogonal transformation on $\Sph^{d-1}$ induced by the action of $R$. This implies (see, e.g., ~\cite{zbMATH03380964} for details) that for any $S\in\mathcal{O}_d$, the volume element $\textrm{d}\mu_d(S)$ is given in the coordinates $(S', \bm{v})$ by 
\begin{equation}\label{volelemtdecomp}
\textrm{d}\mu_d(S)\, = \, \frac{\textrm{d}\bm{v}}{A_d}\cdot\textrm{d}\mu_{d-1}(S')
\end{equation} 
(recall that $\textrm{d}\bm{v}/A_d$ is the volume element of the uniform probability measure on the unit sphere).

Consider now the immersion $$\iota~: \bm{x}\in\R^{d-1}\:\mapsto\: \transp{\left(\transp{\bm{x}}, 0\right)}\in\R^d.$$ Let $P=(P', \bm{w})\in\mathcal{O}_{d}$ (with $P'\in\mathcal{O}_{d-1}$ and $\bm{w}\in\Sph^{d-1}$). %and $\transp{(\transp{\bm{m}}, k)}\in\Z^d$. 
It is then easily seen that $$P\cdot\Z^d\, = \, \Z\bm{w} + f(\bm{w})\cdot \iota\left( P'\cdot \Z^{d-1}\right) \, \supset \, f(\bm{w})\cdot \iota\left( P'\cdot \Z^{d-1}\right).$$
This implies that
\begin{align*}
& p_d\left(\mathcal{E}(\sqrt{\delta}\Delta)\right)\, \underset{\eqref{volelemtdecomp}}{=}\\ &
\frac{1}{A_d}\cdot\int_{\Sph^{d-1}}\textrm{d}\bm{w}\cdot \mu_{d-1}\left(\left\{P'\in\mathcal{O}_{d-1}\; : \; \left(\Z\bm{w} + f(\bm{w})\cdot \iota\left( P'\cdot \Z^{d-1}\right)  \right)\cap \mathcal{E}_d\left(\sqrt{\delta}\Delta\right)\, \neq\, \{\bm{0}\} \right\} \right)\\
&\quad \ge \, \frac{1}{A_d}\cdot\int_{\Sph^{d-1}}\textrm{d}\bm{w}\cdot \mu_{d-1}\left(\left\{P'\in\mathcal{O}_{d-1}\; : \; \left(f(\bm{w})\cdot \iota\left( P'\cdot \Z^{d-1}\right)  \right)\cap \mathcal{E}_d\left(\sqrt{\delta}\Delta\right)\, \neq\, \{\bm{0}\} \right\} \right)\\
&\quad = \, \frac{1}{A_d}\cdot\int_{\Sph^{d-1}}\textrm{d}\bm{w}\cdot \mu_{d-1}\left(\left\{P'\in\mathcal{O}_{d-1}\; : \; P'\cdot \Z^{d-1}\cap \mathcal{E}_d^{(\bm{w})}\left(\sqrt{\delta}\Delta\right)\, \neq\, \{\bm{0}\} \right\} \right),
\end{align*}
where $$\mathcal{E}_d^{(\bm{w})}\left(\sqrt{\delta}\Delta\right) \, := \, \iota^{-1}\left( f(\bm{w})^{-1}\cdot \mathcal{E}_d\left(\sqrt{\delta}\Delta\right)\right).$$
Since the set $\mathcal{E}_d\left(\sqrt{\delta}\Delta\right)\cap \bm{w}^{\perp}$ is sent to $\mathcal{E}_d^{(\bm{w})}\left(\sqrt{\delta}\Delta\right)$ by the linear isomorphism $\bm{x}\in\bm{w}^{\perp}\mapsto \iota^{-1}\left(f(\bm{w})^{-1}\cdot\bm{x} \right)$ which preserves $\mu_{d-1}$--volumes, one obtains that $$p_d\left(\mathcal{E}(\sqrt{\delta}\Delta)\right)\,\ge \, \int_{\Sph^{d-1}} \frac{\textrm{d}\bm{v}}{A_d} \cdot p_{d-1}\left(\mathcal{E}_d(\sqrt{\delta}\Delta)\cap\bm{v}^{\perp}\right).$$ This concludes the proof of Theorem~\ref{thmprinci}.
\end{proof}

\subsection{Proof of Proposition~\ref{propinterellipse}}\label{secellemt} The proof of Proposition~\ref{propinterellipse} is rather elementary and will be done in two steps.

We first seek to prove~\eqref{interellipse}. To this end, it will be convenient to use the Kronecker symbol $\delta_{ij}$ which is equal to 1 if the integers $i$ and $j$ are equal and zero otherwise. Then, with the notation of Proposition~\ref{propinterellipse}, given $\bm{x}=(x_1, \dots, x_d)\in\R^d$, 
\begin{align*}
\bm{x}\in \mathcal{E}_{d-1}\left(\bm{\alpha}, \bm{v}\right) \, &\iff \, \left(\sum_{i=1}^{d}\left( \frac{x_1}{\alpha_i}\right)^2\le 1\right) \wedge \left(x_d=\frac{-1}{v_d}\cdot \sum_{i=1}^{d-1}x_i v_i \right) \\
&\iff \, \frac{1}{\left(v_d\cdot \alpha_d\right)^2}\cdot \left(\sum_{i=1}^{d-1}x_i v_i \right)^2 + \sum_{i=1}^{d-1}\left(\frac{x_i}{\alpha_i}\right)^2 \, \le \, 1 \\ 
&\iff \, \sum_{1\le i, j \le d-1}\left( \frac{\delta_{ij}}{\alpha_i^2}+\frac{v_i v_j}{(v_d\cdot\alpha_d)^2}\right)x_ix_j\, \le \, 1\\
&\iff \, \transp{\bm{y}}\cdot Q \cdot \bm{y} \, \le \, 1,
\end{align*}
where $\bm{y}=\transp{(x_1, \dots, x_{d-1})}\in\R^{d-1}$ and where the matrix $Q$ is defined  in~\eqref{interellipsebis}. Since $Q$ is clearly definite positive, this establishes the first claim in Proposition~\ref{propinterellipse}.

To prove the second claim, denote by $R_{\bm{v}}\in SO_d(\R)$ a rotation in $\R^d$ which maps the first vector $\bm{e_1}$ in the standard basis of $\R^d$ to $\bm{v}$. Let furthermore $Q_{\bm{\alpha}}:=(\alpha_1^{-2}, \dots, \alpha_d^{-2})\in\mathcal{D}_d^{++}$. Then, the $d$--dimensional ellipsoid $\mathcal{E}_{d}\left(\bm{\alpha}\right)$ is congruent to the ellipsoid $$\widetilde{\mathcal{E}}_{d}^{(\bm{v})}(\bm{\alpha})\, := \, \left\{\bm{x}\in\R^d\; : \; \transp{\bm{x}}\cdot \left( \transp{R_{\bm{v}}}Q_{\bm{\alpha}}R_{\bm{v}}\right)\cdot\bm{x}\le 1 \right\}$$ and the $(d-1)$--dimensional ellipsoid $\mathcal{E}_{d-1}\left(\bm{\alpha}, \bm{v}\right) $ becomes congruent to the ellipsoid $\widetilde{\mathcal{E}}_{d}^{(\bm{v})}(\bm{\alpha})\cap \left\{ x_1=0\right\}$ given by a positive definite matrix $Q_{\bm{\alpha}}^{(\bm{v})}\in\mathcal{S}_{d-1}^{++}$. This matrix $Q_{\bm{\alpha}}^{(\bm{v})}$ is obtained by stripping off the matrix $ \transp{R_{\bm{v}}}Q_{\bm{\alpha}}R_{\bm{v}}$ from its first row and first column. Let $\beta_{d-1}^{-2}\le \dots\le \beta_{1}^{-2}$ denote the eigenvalues of $Q_{\bm{\alpha}}^{(\bm{v})}$ (in other words, $\beta_1, \dots, \beta_{d-1}$ are the lengths of the semi--principal axes of the ellipsoid $\widetilde{\mathcal{E}}_{d}^{(\bm{v})}(\bm{\alpha})\cap \left\{ x_1=0\right\}$). It then follows from a direct application of the Cauchy Interlacing Inequalities that $$\frac{1}{\alpha_d^2}\, \le \, \frac{1}{\beta_{d-1}^2}\, \le \dots \le \, \frac{1}{\beta_1^2}\, \le \, \frac{1}{\alpha_1^2},$$ which completes the proof of Proposition~\ref{propinterellipse}.

\subsection{Proof of Proposition~\ref{propformulerecu}} Before proving Proposition~\ref{propformulerecu}, we make a crucial remark which will be used several times hereafter. Fix $\bm{\alpha}\in\R^d$ satisfying~\eqref{inegaslpha}. Let 
\begin{equation}\label{defAdell}
\mathcal{A}_d(\bm{\alpha})\, :=\, \mathcal{E}_d(\bm{\alpha})\cap\Sph^{d-1}
\end{equation}
and $\bm{x}:=(x_1, \dots, x_d)\in\R^d$. Then, 
\begin{align*}
\bm{x}\in \mathcal{A}_d(\bm{\alpha}) \, &\iff \, \left(\sum_{i=1}^{d}\left( \frac{x_1}{\alpha_i}\right)^2\le 1\right) \wedge \left(\sum_{i=1}^{d}x_i^2=1\right) \\
&\iff \, \left(\sum_{i=1}^{d-1}x_i^2\cdot\left(\frac{1}{\alpha_i^2} -\frac{1}{\alpha_d^2}\right)\le 1-\frac{1}{\alpha_d^2}\right) \wedge \left(\sum_{i=1}^{d}x_i^2=1\right).
\end{align*}
Given $\bm{\mu}\in (\R_{>0})^{d-1}$, let $\mathcal{C}_d(\bm{\mu})$ denote the \emph{full} cylinder with axis spanned by $\bm{e_d}$ whose section with the hyperplane $\left\{ x_d=0\right\}$ is the $(d-1)$--dimensional ellipsoid $\mathcal{E}_{d-1}(\bm{\mu})$. With the notation of Proposition~\ref{propformulerecu}, the above chain of equivalences thus amounts to claiming that 
\begin{equation}\label{defcyl}
\mathcal{A}_d(\bm{\alpha})\, =\, \mathcal{C}_d(\utilde{\bm{\alpha}})\cap\Sph^{d-1}.
\end{equation}

\begin{proof}[Proof of Proposition~\ref{propformulerecu}]
Note first that the relations~\eqref{castriviaux} are trivial. Indeed, under~\eqref{inegaslpha}, $\mathcal{A}_d(\bm{\alpha}) = \Sph^{d-1}$ if $\alpha_1\ge 1$ and $\# \mathcal{A}_d(\bm{\alpha})\le 2$ if $\alpha_d\le 1$. Assume therefore that $\alpha_1<1<\alpha_d$. Parameter a dense open set in $\Sph^{d-1}$ as follows~: $$\bm{v}\, = \, \left(\bm{u}\cdot \sin \theta, \, \cos \theta\right),$$ where $\bm{u}\in\Sph^{d-2}$ and $\theta\in (0, \, \pi)$ ($\theta$ is thus the angle between $\bm{u}$ and $\bm{e_d}$). A standard calculation shows that, in these coordinates, the volume element $\textrm{d}\bm{v}$ reads $\textrm{d}\bm{v} = (\sin \theta)^{d-2}\cdot\textrm{d}\theta \cdot\textrm{d}\bm{u}$ (if $d=2$, $\textrm{d}\bm{u}$ is the counting probability measure on $\Sph^0=\{ \pm 1\}$). Therefore, $$\osig_{d-1}\left(\mathcal{E}_d(\bm{\alpha}) \right)\, = \, \frac{1}{A_d}\int_{0}^{\pi}\textrm{d}\theta\cdot\left(\sin \theta\right)^{d-2}\int_{\Sph^{d-2}}\chi_{\left[\left(\bm{u}\cdot \sin \theta, \, \cos \theta\right)\, \in \, \mathcal{A}_d(\bm{\alpha}) \right]}\cdot \textrm{d}\bm{u}.$$
 In view of~\eqref{defAdell} and~\eqref{defcyl}, the intersection of $\mathcal{A}_d(\bm{\alpha})$ with the hyperplane $\left\{x_d=\cos \theta \right\}$ is obtained as the intersection of the $(d-1)$--dimensional ellipsoid $\mathcal{E}_{d-1}(\utilde{\bm{\alpha}})$ with the $(d-1)$--dimensional unit sphere centred at the origin with radius $\sin \theta$~:
$$\bm{x}\in \mathcal{A}_d(\bm{\alpha})\cap \left\{x_d=\cos \theta \right\}\, \iff \, \left(\sum_{i=1}^{d-1}\left( \frac{x_1}{\utilde{\alpha}_i}\right)^2\le 1\right) \wedge \left(\sum_{i=1}^{d-1}x_i^2=\sin^2\theta\right) \wedge \left(x_d=\cos\theta\right).$$ This implies that, given $\bm{u}\in\Sph^{d-2}$ and $\theta\in (0, \pi)$, $$\left(\bm{u}\cdot \sin \theta, \, \cos \theta\right)\, \in \, \mathcal{A}_d(\bm{\alpha}) \, \iff \, \bm{u}\, \in \, \mathcal{E}_{d-1}\left(\frac{\utilde{\bm{\alpha}}}{\sin\theta}\right). $$ 
Thus~:
\begin{align*}
\osig_{d-1}\left(\mathcal{E}_d(\bm{\alpha}) \right)\, &= \, \frac{1}{A_d}\int_{0}^{\pi}\textrm{d}\theta\cdot\left(\sin \theta\right)^{d-2}\int_{\Sph^{d-2}}\chi_{\left[\bm{u}\, \in \,  \mathcal{E}_{d-1}\left(\frac{\bm{\utilde{\alpha}}}{\sin\theta}\right) \right]}\cdot \textrm{d}\bm{u}\\
&=\, \frac{A_{d-1}}{A_d}\cdot \int_{0}^{\pi}\textrm{d}\theta\cdot\left(\sin \theta\right)^{d-2}\cdot \osig_{d-2}\left( \mathcal{E}_{d-1}\left(\frac{\bm{\utilde{\alpha}}}{\sin\theta}\right)\right).
\end{align*}
The result then follows from~\eqref{defVdAd} and~\eqref{wallis}.
\end{proof}

\subsection{Proof of Proposition~\ref{propestimineg}} The proof of Proposition~\ref{propestimineg} rests on the following lemma. Throughout, we adopt the notation introduced before the statement of Proposition~\ref{propestimineg} and fix $\bm{\alpha}\in\R^d$ satisfying~\eqref{inegaslpha} and the inequalities $\alpha_{1}< 1 < \alpha_d$. Let furthermore $$K_d(\bm{\alpha}) \, := \, \prod_{i=1}^{d} \left[-\alpha_i, \, \alpha_i \right].$$

\begin{lem}\label{messphcube}
The following equation holds~: $$\osig_{d-1}\left(K_d(\bm{\alpha}) \right)\, = \, \mathfrak{I}_d\left(\bm{\alpha}\right).$$ Furthermore, one has also the estimates $$ \mathfrak{L}_d\left(\bm{\alpha} \right) \cdot\left(\frac{2}{\pi}\right)^{d-2}\, \le \, \mathfrak{I}_d\left(\bm{\alpha}\right)\, \le \, \mathfrak{L}_d\left(\bm{\alpha} \right)$$ with $$\mathfrak{L}_d\left(\bm{\alpha} \right)\, := \, \frac{2^{d}}{(d-1)!\cdot A_d}\cdot\prod_{j=1}^{d-1}\left(\left( \frac{\pi}{2}\right)^j - b\left(\alpha_{d-j} \right)^j  \right).$$
\end{lem}

\begin{proof}
Parametrise the unit sphere in spherical coordinates by defining the coordinates of $\bm{v}:=\bm{v_{d}}\in\Sph^{d-1}$ by induction in the following way~: $$\bm{v_d}\, = \, \left(\cos \theta_1, \,\bm{v_{d-1}}\cdot \sin \theta_1  \right),$$ where $\bm{v_{k}}\in\Sph^{k-1}$ for $k=2, \dots, d-1$. Here, the base case is $\bm{v_{2}} = \left(\cos \theta_{d-1}, \, \sin \theta_{d-1} \right)\in \Sph^{1}$. Thus, given $i= 1, \dots, d-1$, the real number $\theta_i$ is the angle between $\bm{v}$ and the $i^{th}$ standard vector basis $\bm{e_i}$ of $\R^d$. These angles $\theta_i$ are unique upon requiring that $\theta_i\in [0, \pi]$ for $i=1, \dots, d-2 $ and $\theta_{d-1}\in [0, 2\pi)$. Upon taking into account the notation convention adopted here to label the angles, the volume element $\textrm{d}\bm{v}$ is then given by the usual formula $$\textrm{d}\bm{v}\, = \, \frac{1}{A_d}\cdot\prod_{i=2}^{d}\sin^{i-2}\theta_{d-i+1}\cdot \textrm{d}\theta_{d-i+1}.$$ Thus, given $\bm{v}\in\R^d$ with (cartesian) coordinates $(x_1, \dots, x_d)$,
\begin{align*}
\bm{v}\in K_d(\bm{\alpha}) \cap \Sph^{d-1} \, &\iff \, \forall i\in \llbracket 1, d\rrbracket, \; \left|x_i \right| = \left|\cos \theta_i \right|\le \alpha_i\\
&\underset{(\alpha_d>1)}{\iff} \, \forall i\in \llbracket 1, d-1\rrbracket, \; \left|\cos \theta_i \right|\le \alpha_i\\ 
&\iff \, 
\left\{
    \begin{array}{ll}
        \forall i\in \llbracket 1, d-2\rrbracket, \; \theta_i\in\left[b(\alpha_i), \, \pi-b(\alpha_i) \right],\\
        \theta_{d-1}\in \left[b(\alpha_{d-1}), \, \pi-b(\alpha_{d-1}) \right]\cup \left[\pi +b(\alpha_{d-1}), \, 2\pi-b(\alpha_{d-1}) \right] 
    \end{array}
\right.
\end{align*}
(with obvious changes for the bounds of the latter intervals when $b(\alpha_{d-1})=0$). Therefore,
\begin{align*}
\osig_{d-1}\left(K_d\left(\bm{\alpha} \right) \right)\, & = \, \frac{1}{A_d}\cdot 2(\pi-2b(\alpha_{d-1}))\cdot\prod_{i=3}^{d} \int_{b(\alpha_{d-i+1})}^{\pi- b(\alpha_{d-i+1})} \sin^{i-2}\theta\cdot \textrm{d} \theta\\
&=\, \frac{2^{d}}{A_d}\cdot\left(\frac{\pi}{2}-b(\alpha_{d-1})\right)\cdot \prod_{i=3}^{d}\int_{b(\alpha_{d-i+1})}^{\pi/2} \sin^{i-2}\theta\cdot \textrm{d} \theta\\
&\underset{\eqref{defIdalspha}}{=}\, \mathfrak{I}_d(\bm{\alpha}).
\end{align*}
The estimates involving $\mathfrak{L}_d(\bm{\alpha})$ follow now straightforwardly from the definition of $\mathfrak{I}_d(\bm{\alpha})$ and from the convexity inequalities $(2/\pi)\cdot t \le \sin t \le t$ valid for any $t\in [0, \pi/2]$.
\end{proof}

\begin{proof}[Proof of Proposition~\ref{propestimineg}]
It plainly follows from the definition of the ellipsoid $\mathcal{E}_d\left(\mathcal{\bm{\alpha}} \right)$ in~\eqref{defellipsoide} that 
\begin{equation}\label{rectellip1}
\prod_{i=1}^{d}\left[-\frac{\alpha_i}{\sqrt{d}}, \, \frac{\alpha_i}{\sqrt{d}} \right]\; \subset \; \mathcal{E}_d\left(\mathcal{\bm{\alpha}} \right) \; \subset \;  \prod_{i=1}^{d}\left[-\alpha_i, \, \alpha_i \right].
\end{equation}
Also, relations \eqref{defAdell} and~\eqref{defcyl} imply that 
\begin{equation}\label{rectellip2}
\left(\prod_{i=1}^{d-1}\left[-\frac{\utilde{\alpha}_i^*}{\sqrt{d-1}}, \, \frac{\utilde{\alpha}_i^*}{\sqrt{d-1}} \right]\right) \times [-1, \, 1]\; \subset \; \mathcal{A}_d\left(\mathcal{\bm{\alpha}} \right) \; \subset \;  \left(\prod_{i=1}^{d-1}\left[-\utilde{\alpha}_i^*, \, \utilde{\alpha}_i^* \right]\right)\times [-1, \, 1]
\end{equation}
(this is because the basis of the cylinder $\mathcal{C}_d\left(\mathcal{\bm{\utilde{\alpha}}} \right)$ is the ellipsoid $\mathcal{E}_{d-1}\left(\mathcal{\bm{\utilde{\alpha}}} \right)$). 

Thus, the estimates for $\osig_{d-1}\left(\mathcal{E}_d\left(\mathcal{\bm{\alpha}} \right) \right)$ in Proposition~\ref{propestimineg} become straightforward consequences of relations~\eqref{rectellip1} and~\eqref{rectellip2} and of Lemma~\ref{messphcube}. As for the bounds for $\mathfrak{I}_d\left(\bm{\alpha}\right)$ therein, they also follow from Lemma~\ref{messphcube} and from the inequalities
\begin{align*}
\left( \frac{\pi}{2}\right)^{j-1}\cdot \min\left\{1, \alpha_{d-j} \right\}\, \le \,\left( \frac{\pi}{2}\right)^j- b(\alpha_{d-j})^j\, \le \, j\cdot \left( \frac{\pi}{2}\right)^{j}\cdot \min\left\{1, \alpha_{d-j} \right\}.
\end{align*}
The latter is a direct consequence of the convexity inequalities $$x\le \frac{\pi}{2} - \arccos x \le \frac{\pi}{2} x$$ valid for all $x\in [0, 1]$ and of the factorisation identity $$\left( \frac{\pi}{2}\right)^j- b(\alpha_{d-j})^j\, = \, \left(\frac{\pi}{2} - b(\alpha_{d-j}) \right)\cdot \sum_{k=0}^{j-1} \left( \frac{\pi}{2}\right)^{j-1-k}b(\alpha_{d-j})^k.$$
\end{proof}

\subsection{Proof of Theorem~\ref{thmapplicunif}} Let $\epsilon>0$ and let $\Delta:=\left(\alpha_1, \dots, \alpha_d\right)\in\Delta^{++}_d$ be such that the vector $\bm{\alpha'}:=\left(\alpha_1, \dots, \alpha_{d-1}\right)$ lies in the support of the measure $\nu_d^{(\epsilon)}$ as defined in~\eqref{defvdepsi} (i.e.~$\epsilon\le \alpha_i\le \epsilon^{-1}$ for all $i=1, \dots, d-1$). This clearly implies that $\left\|\Delta\right\|_{\infty}\le \epsilon^{-d+1}$. In particular, in view of the upper bound in~\eqref{resultprinci}, the probability $\tau^{(\epsilon)}_d\left(\mathfrak{F}(\delta)\right)$ vanishes whenever $\sqrt{\delta}\cdot \epsilon^{-d+1}<1$, i.e.~whenever $\delta<\epsilon^{2\cdot(d-1)}$. Since $\nu_d^{(\epsilon)}\left(\Delta^{++}_d\backslash\Delta^{++}_{d, sub} \right)=0$, the same conclusion holds if $\delta=\epsilon^{2\cdot(d-1)}$. This establishes~\eqref{probanulle}.

Assume from now on that $\delta> \epsilon^{2(d-1)}$. The goal is to bound from below and above the probability $$\tau_d^{(\epsilon)}\left(\mathfrak{F}_d(\delta) \right)\, = \, \frac{1}{\left|2\log \epsilon \right|^{d-1}}\cdot  \int_{[\epsilon, \, \epsilon^{-1}]^{d-1}} \prod_{i=1}^{d-1}\frac{\textrm{d}\alpha_i}{\alpha_i}\cdot p_d\left(\mathcal{E}\left( \sqrt{\delta}\Delta\right) \right) .$$ Upon reordering the  coordinates of the vector $\Delta$ as defined above, it follows from the invariance of the quantity $p_d\left(\mathcal{E}\left( \sqrt{\delta}\Delta\right) \right)$ under such permutation that 
\begin{align*}
\frac{(d-1)!}{\left|2\log \epsilon \right|^{d-1}}\cdot & \int_{\underset{\alpha_{d-1}<\alpha_d := \left(\alpha_1\dots \alpha_{d-1} \right)^{-1}}{\epsilon \le \alpha_1 < \dots < \alpha_{d-1}\le \epsilon^{-1}}} \prod_{i=1}^{d-1}\frac{\textrm{d}\alpha_i}{\alpha_i}\cdot p_d\left(\mathcal{E}\left( \sqrt{\delta}\Delta\right) \right)\,\le \tau_d^{(\epsilon)}\left(\mathfrak{F}_d(\delta) \right)\\
&\le \, \frac{d!}{\left|2\log \epsilon \right|^{d-1}}\cdot\int_{\underset{\alpha_{d-1}<\alpha_d := \left(\alpha_1\dots \alpha_{d-1} \right)^{-1}}{\epsilon \le \alpha_1 < \dots < \alpha_{d-1}\le \epsilon^{-1}}} \prod_{i=1}^{d-1}\frac{\textrm{d}\alpha_i}{\alpha_i}\cdot p_d\left(\mathcal{E}\left( \sqrt{\delta}\Delta\right) \right).
\end{align*}
Here, we are using two facts to obtain the upper bound~: on the one hand, if $\sigma$ is a permutation of $\llbracket 1, \, d\rrbracket$ such that, given a $d$--tuple $\left(\alpha_1, \dots, \alpha_d\right)$, $\alpha_{\sigma(1)}\le \dots \le \alpha_{\sigma(d)}$, then $\prod_{i=1}^{d-1}\alpha_i^{-1}\le \prod_{i=i}^{d-1}\alpha_{\sigma(i)}^{-1}$; on the other, given a $d$-tuple $\left(\beta_1, \dots, \beta_d\right)$ such that $\beta_1<\dots < \beta_d$, there are $d!$ $d$--tuples $\left(\alpha_1, \dots, \alpha_d\right)$ for which there exists a permutation $\sigma$ such that $\alpha_{\sigma(1)}=\beta_1, \dots, \alpha_{\sigma(d)}=\beta_d$. The lower bound follows from a similar argument~:  given a $d$-tuple $\left(\beta_1, \dots, \beta_d\right)$ such that $\beta_1<\dots < \beta_d$, there are $(d-1)!$ $d$--tuples $\left(\alpha_1, \dots, \alpha_d\right)$ for which there exists a permutation $\sigma$ of $\llbracket 1, d-1 \rrbracket$ such that $\alpha_{\sigma(1)}=\beta_1, \dots, \alpha_{\sigma(d-1)}=\beta_{d-1}$ and $\alpha_d = \beta_d = \max_{1\le i \le d} \beta_i$. 

Note that in the domain of integration, 
\begin{equation}\label{norminfdomint}
\left\|\Delta \right\|_{\infty}= \alpha_d= \left(\alpha_1\dots \alpha_{d-1} \right)^{-1}.
\end{equation}  
Since  from Proposition~\ref{propformulerecu}, $p_d\left(\mathcal{E}\left( \sqrt{\delta}\Delta\right) \right)=0$ whenever $\sqrt{\delta}\cdot \alpha_d \le 1$, one has also 
\begin{align}
\frac{(d-1)!}{\left|2\log \epsilon \right|^{d-1}}\cdot & \int_{\underset{\max\left\{\delta^{-1/2},\, \alpha_{d-1}\right\}<\left(\alpha_1\dots \alpha_{d-1} \right)^{-1}}{\epsilon \le \alpha_1 < \dots < \alpha_{d-1}\le \epsilon^{-1}}} \prod_{i=1}^{d-1}\frac{\textrm{d}\alpha_i}{\alpha_i}\cdot p_d\left(\mathcal{E}\left( \sqrt{\delta}\Delta\right) \right)\,\le \tau_d^{(\epsilon)}\left(\mathfrak{F}_d(\delta) \right) \label{thmappliprepazero}\\
&\le \, \frac{d!}{\left|2\log \epsilon \right|^{d-1}}\cdot\int_{\underset{\max\left\{\delta^{-1/2},\, \alpha_{d-1}\right\}< \left(\alpha_1\dots \alpha_{d-1} \right)^{-1}}{\epsilon \le \alpha_1 < \dots < \alpha_{d-1}\le \epsilon^{-1}}} \prod_{i=1}^{d-1}\frac{\textrm{d}\alpha_i}{\alpha_i}\cdot p_d\left(\mathcal{E}\left( \sqrt{\delta}\Delta\right) \right).\label{thmappliprepa}
\end{align} 

We now call on Theorem~\ref{thmprinci} to bound the probability $p_d\left(\mathcal{E}\left( \sqrt{\delta}\Delta\right) \right)$ as follows~: 
\begin{align}
\osig_{d-1} \left( \mathcal{E}_d(\sqrt{\delta}\Delta)\right) \, \le \, p_d\left(\mathcal{E}\left( \sqrt{\delta}\Delta\right) \right) \, \le \, \sum_{\underset{\left\|\bm{n} \right\|_{\infty}\le \sqrt{\delta}\cdot \left\|\Delta \right\|_{\infty}}{\bm{n}\in\Z^d\backslash \left\{\bm{0} \right\}}}\osig_{d-1}\left(\mathcal{E}\left( \frac{\sqrt{\delta}}{\left\| \bm{n}\right\|_2}\cdot \Delta\right) \right). \label{thmapplithm2}
\end{align} 
Furthermore, from Proposition~\ref{propestimineg}, 
\begin{equation}\label{thmappliavtdern}
\osig_{d-1} \left( \mathcal{E}_d(\sqrt{\delta}\Delta)\right) \, \ge \, a(d)\cdot  \prod_{i=1}^{d-1}\min\left\{\frac{\sqrt{\delta}\cdot \alpha_i}{d}, \, 1 \right\}\, \ge \, \frac{a(d)}{d^{d-1}}\cdot  \prod_{i=1}^{d-1}\min\left\{\sqrt{\delta}\cdot \alpha_i, \, 1 \right\}.
\end{equation} 
Given the domain of integration of the integrals above, one has also
\begin{align}
\sum_{\underset{\left\|\bm{n} \right\|_{\infty}\le \sqrt{\delta}\cdot \left\|\Delta \right\|_{\infty}}{\bm{n}\in\Z^d\backslash \left\{\bm{0} \right\}}}\osig_{d-1}\left(\mathcal{E}\left( \frac{\sqrt{\delta}}{\left\| \bm{n}\right\|_2}\cdot \Delta\right) \right)\, &\le \, \sum_{\underset{\left\|\bm{n} \right\|_{\infty}\le \sqrt{\delta}\cdot \left\|\Delta \right\|_{\infty}}{\bm{n}\in\Z^d\backslash \left\{\bm{0} \right\}}}a'(d)\cdot \prod_{i=1}^{d-1}\min\left\{\frac{\sqrt{\delta}\cdot \alpha_i}{\left\|\bm{n}\right\|_2},\, 1 \right\} \nonumber \\
&\le \, a'(d)\cdot \delta^{(d-1)/2}\cdot \left(\prod_{i=1}^{d-1}\alpha_i\right)\cdot \left( \sum_{\underset{\left\|\bm{n} \right\|_{\infty}\le \sqrt{\delta}\cdot \left\|\Delta \right\|_{\infty}}{\bm{n}\in\Z^d\backslash \left\{\bm{0} \right\}}}\frac{1}{\left\|\bm{n} \right\|_{\infty}^{d-1}}\right) \nonumber\\
&\le \, a'(d)\cdot \delta^{(d-1)/2}\cdot \left(\prod_{i=1}^{d-1}\alpha_i\right)\cdot \left( \sum_{k=1}^{\sqrt{\delta}\cdot \left\|\Delta \right\|_{\infty}}d\cdot \frac{(2k+1)^{d-1}}{k^{d-1}}\right) \nonumber\\
&\le \, a'(d)\cdot \delta^{(d-1)/2}\cdot \left(\prod_{i=1}^{d-1}\alpha_i\right)\cdot \left(3^{d-1}\cdot d\cdot \sqrt{\delta}\cdot  \left\|\Delta \right\|_{\infty}\right) \nonumber\\
&\underset{\eqref{norminfdomint}}{\le} 3^{d-1}\cdot a'(d)\cdot d\cdot \delta^{d/2}.
\label{majosigthmappli}
\end{align}

Inequalities~\eqref{inegprobaepsi} thus turn out to be a rephrasing of the relations~\eqref{thmappliprepazero}---\eqref{majosigthmappli} with the constants $c_d(\epsilon)$ and $C_d(\epsilon)$ stated in the theorem.

As for inequalities~\eqref{inegs_d1} and~\eqref{inegs_d2}, note first that, on the one hand,
\begin{align*}
s_d(\epsilon, \delta)\, \ge \, \min\left\{\sqrt{\delta}, \epsilon \right\}^{d-1}\cdot \int_{\underset{\max\left\{\epsilon^{-1}, \delta^{-1/2}\right\}<\left(\alpha_1\dots \alpha_{d-1} \right)^{-1}}{\epsilon\le \alpha_1<\dots <\alpha_{d-1}\le \epsilon^{-1}}} \textrm{d}\alpha_1 \dots \textrm{d}\alpha_{d-1}
\end{align*}
and that, on the other, 
\begin{align*}
S_d(\epsilon, \delta)\, \le \, \delta^{d/2}\cdot \int_{\underset{\delta^{-1/2}<\left(\alpha_1\dots \alpha_{d-1} \right)^{-1}}{\epsilon\le \alpha_1<\dots <\alpha_{d-1}\le \epsilon^{-1}}} \prod_{i=1}^{d-1}\frac{\textrm{d}\alpha_i}{\alpha_{i}}\cdotp
\end{align*}

Now, given any $c>0$, the change of variables $y_i=\alpha_i$ for $1\le i \le d-2$ and $y_{d-1}=\prod_{i=1}^{d-1}\alpha_i$ shows that 
\begin{align*}
\int_{\underset{c<\left(\alpha_1\dots \alpha_{d-1} \right)^{-1}}{\epsilon\le \alpha_1<\dots <\alpha_{d-1}\le \epsilon^{-1}}} \textrm{d}\alpha_1 \dots \textrm{d}\alpha_{d-1}\, & = \, \int_{\underset{\epsilon^{d-1}<y_{d-1}<c^{-1}}{\epsilon \le y_1<\dots <y_{d-2}\le \epsilon^{-1}}}\textrm{d}y_{d-1}\cdot \prod_{i=1}^{d-2}\frac{\textrm{d}y_{i}}{y_i}\\
& =\, \frac{\left|2\log \epsilon \right|^{d-2}}{(d-2)!}\cdot \left(c^{-1}-\epsilon^{d-1} \right)
\end{align*}
and that 
\begin{align*}
\int_{\underset{c<\left(\alpha_1\dots \alpha_{d-1} \right)^{-1}}{\epsilon\le \alpha_1<\dots <\alpha_{d-1}\le \epsilon^{-1}}} \prod_{i=1}^{d-1}\frac{\textrm{d}\alpha_i}{\alpha_{i}} \, & = \, \int_{\underset{\epsilon^{d-1}<y_{d-1}<c^{-1}}{\epsilon \le y_1<\dots <y_{d-2}\le \epsilon^{-1}}}\prod_{i=1}^{d-1}\frac{\textrm{d}y_{i}}{y_i}\\
& =\, \frac{\left|2\log \epsilon \right|^{d-2}}{(d-2)!}\cdot \log\left(\frac{c^{-1}}{\epsilon^{d-1}} \right).
\end{align*}
This completes the proof of Theorem~\ref{thmapplicunif}.
 
\section{An Approach via the Cholesky Decomposition.}\label{approachcholesky}

The probabilistic approach via the spectral decomposition exposed in \S\ref{approachspectral} requires that the probability measures under consideration be essentially defined from the set of eigenvalues of a given element in $\Sigma_{d}^{++}$. While this should not be seen as a big restriction in view of the spectral decomposition and of the fact that the orthogonal group is compact, the determination of the eigenvalues of a matrix is known to be a hard task. We therefore adopt here an alternative approach based on the Cholesky decomposition of a quadratic form in $\Sigma_{d}^{++}$ or, in view of Problem~\ref{pb2}, on the Cholesky decomposition of a quadratic form in $\mathcal{S}_{d}^{++}$. 

Let $\mathcal{T}_d^{++}$ be the group of upper triangular matrices with strictly positive diagonal entries. Let $\Theta_d^{++}$ be the subgroup of $\mathcal{T}_d^{++}$ consisting of all those matrices with determinant one~: 
\begin{equation}\label{deftheta++}
\Theta_d^{++}:=\mathcal{T}_d^{++}\cap SL_d(\R).
\end{equation} 
Let 
\begin{equation}\label{defp}
p\, := \, \frac{d(d-1)}{2}\cdotp
\end{equation}  
The set $\mathcal{T}_d^{++}$ shall be identified with $\left(\R_{>0} \right)^{d}\times \R^{p}$ by splitting a matrix therein between its $d$ diagonal terms and the remaining $p$ off--diagonal upper coefficients. A generic element in $\mathcal{T}_d^{++}$ shall thus be represented as $(\bm{\beta}, \bm{u})$ with $\bm{\beta}\in \left(\R_{>0} \right)^{d}$ and $\bm{u}\in\R^{p}$, in which case it will be convenient to adopt the notation $$\bm{\beta}\,:=\, (\beta_1, \bm{\utilde{\beta}})$$ with $\beta_1\in\R$ and $\bm{\utilde{\beta}}\in\R^{d-1}$ (this notation is independent from~\eqref{quiproquo}). In the same way, the set $\Theta_d^{++}$ shall be identified with $\left(\R_{>0} \right)^{d-1}\times \R^{p}$. A generic element of $\Theta_d^{++}$ shall thus be represented as $(\bm{\beta'}, \bm{u})$ with $\bm{\beta'}\in \left(\R_{>0} \right)^{d-1}$ and $\bm{u}\in\R^{p}$, in which case it will be convenient to adopt the notation $$\bm{\beta'}\,:=\, (\beta'_1, \bm{\utilde{\beta}'})$$ with $\beta'_1\in\R$ and $\bm{\utilde{\beta'}}\in\R^{d-2}$. When a matrix in $\Theta_d^{++}$ is seen as an element of $\mathcal{T}_d^{++}$, it shall also be given as a vector from $\left(\R_{>0} \right)^{d}\times \R^{p}$. This should not cause any confusion.

The Cholesky decomposition of a positive definite matrix amounts to claiming that the map 
\begin{equation}\label{cholmap}
\varphi_{chol}\; : \; L\in \mathcal{T}_d^{++} \; \mapsto \; \transp{L}L\in \mathcal{S}_d^{++}
\end{equation} 
is bijective. This implies in particular that the map 
\begin{equation}\label{phitildechol}
\widetilde{\varphi}_{chol}\; : \; L\in \Theta_d^{++} \; \mapsto \; \transp{L}L\in \Sigma_d^{++}
\end{equation} 
is also bijective.  Determining the Cholesky decomposition of a given positive definite matrix is a problem which has been extensively studied from an algorithmic point of view and which can be implemented in a very efficient way --- see, e.g., \cite{watkins} for details.

\subsection{Definition of a Suitable Class of Measures}  Note that $\mathcal{S}_d^{++}$ sits as an open cone in the space of symmetric matrices in dimension $d$. It is a $(p+d)$--dimensional manifold (with $p$ as defined in~\eqref{defp}) and any matrix therein can be identified with a vector in $\R^{p+d}$ by considering its upper triangular part. Similarly, $\Sigma_d^{++}$ sits as a $(p+d-1)$--dimensional manifold in $\mathcal{S}_d^{++}$ which can be identified with a subset of $\R^{p+d-1}$ by considering the upper triangular part of a matrix therein minus the bottom right coefficient. For a rigorous justification of the fact that this indeed gives a system of independent coordinates, see (the proof of) Lemma~\ref{lemjac} in \S\ref{pruevcol} below.

With the help of these identifications, we will be concerned with measures supported on $\mathcal{S}_d^{++}$ (resp.~on $\Sigma_d^{++}$) absolutely continuous with respect to the $(p+d)$--dimensional Lebesgue measure $\lambda_{p+d}$ (resp.~with respect to the $(p+d-1)$--dimensional Lebesgue measure $\lambda_{p+d-1}$).

Let then $f\, : \, \mathcal{S}_d^{++} \rightarrow \R_+$ (resp.~$\widetilde{f}\, : \, \Sigma_d^{++} \rightarrow \R_+$) be a density function supported on $\mathcal{S}_d^{++}$ (resp.~on $\Sigma_d^{++}$). The corresponding measure is denoted by $\nu_f$ (resp.~by $\widetilde{\nu}_{\widetilde{f}}$). 

\subsection{The Main Estimates} Given $\delta>0$, the quantities of interest are 
\begin{equation}\label{defprobchol}
m_f\left(\delta \right)\, := \, \nu_f\left(\left\{Q\in \mathcal{S}_d^{++}\; : \; M_d(Q)\, \le \, \delta\right\} \right)
\end{equation} 
and $$\widetilde{m}_{\widetilde{f}}\left(\delta \right)\, := \, \widetilde{\nu}_{\widetilde{f}}\left(\left\{\Sigma\in \Sigma_d^{++}\; : \; M_d(\Sigma)\, \le \, \delta\right\} \right).$$

Given any $\bm{\beta}\in (\R_{>0})^d$, define $$G_f(\bm{\beta})\, := \, 2^d\cdot \prod_{i=1}^{d} \beta_i^{d-i+1}\cdot \int_{\R^p}\left(f\circ \varphi_{chol}\right)\left(\bm{\beta}, \bm{u} \right)\cdot\textrm{d}\lambda_p(\bm{u})$$
and, given any $\beta_1>0$, let
\begin{equation}\label{defgchol}
g_f(\beta_1)\, :=\, \int_{\left(\R_{>0} \right)^{d-1}}G_f(\beta_1, \, \bm{\utilde{\beta}})\cdot \textrm{d}\lambda_{d-1}(\bm{\utilde{\beta}}).
\end{equation}

Similarly, given any $\bm{\beta'}\in (\R_{>0})^{d-1}$, define $$\widetilde{G}_{\widetilde{f}}(\bm{\beta'})\, := \, 2^{d-1}\cdot \prod_{i=1}^{d-1} \beta_i^{d-i+1}\cdot \int_{\R^p}\left(\widetilde{f}\circ \widetilde{\varphi}_{chol}\right)\left(\bm{\beta'}, \bm{u} \right)\cdot\textrm{d}\lambda_p(\bm{u})$$
and, given any $\beta'_1>0$, let $$\widetilde{g}_{\widetilde{f}}(\beta'_1)\, :=\, \int_{\left(\R_{>0} \right)^{d-2}}\widetilde{G}_{\widetilde{f}}(\beta'_1, \, \bm{\utilde{\beta'}})\cdot \textrm{d}\lambda_{d-2}(\bm{\utilde{\beta'}}).$$

With these definitions, the main theorem in this section reads as follows~: 

\begin{thm}\label{thmchol}
Let $\delta\in (0, \, 1)$. Then, 
\begin{equation}\label{cholestim1}
0\; \le \; 1-\int_{\sqrt{\delta}}^{\infty} g_f\; \le \; m_f(\delta) \; \le \; 1-\int_{I_d(\delta)} G_f\; \le \; 1,
\end{equation}
where $$I_d(\delta)\, := \, \left(\sqrt{\delta},\, +\infty\right)^d.$$ 

Furthermore, one has also the estimates
\begin{equation}\label{cholestim2}
0\; \le \; 1-\int_{\sqrt{\delta}}^{\infty} \widetilde{g}_{\widetilde{f}}\; \le \; \widetilde{m}_{\widetilde{f}}(\delta) \; \le \; 1-\int_{\Delta_{d-1}(\delta)} \widetilde{G}_{\widetilde{f}}\; \le \; 1,
\end{equation}
where $$\Delta_{d-1}(\delta)\, :=\, \left\{\bm{\beta'}\in\left(\R_{>0}\right)^{d-1} \; :\; \left(\forall i\in \llbracket 1, \, d-1\rrbracket, \, \beta_i>\sqrt{\delta} \right) \wedge \left(\prod_{i=1}^{d-1} \beta_i<\frac{1}{\sqrt{\delta}} \right)\right\}.$$ 
\end{thm}

Both sets of inequalities~\eqref{cholestim1} and~\eqref{cholestim2} provide non--trivial lower and upper bounds for the probabilities $m_f\left(\delta \right)$ and $\widetilde{m}_{\widetilde{f}}\left(\delta \right)$, although the former bounds are doomed to be cruder than the latter (see the proof in~\S\ref{pruevcol} for details). In fact, we will mostly be interested in obtaining accurate upper bounds. In this respect, it is worth pointing out that those obtained above amount to finding short lattice vectors in a ball with respect to the sup--norm in $\R^d$ centered at the origin rather than in the largest Euclidean ball contained in it (see the proof of Lemma~\ref{veccourt} below for details). For ``not too wild'' density functions, the loss of accuracy in doing so should be seen as involving a multiplicative constant depending only on the dimension $d$.

\subsection{A Numerical Example.} A most standard distribution supported on the set of positive definite matrices is the so--called \emph{Wishart distribution}. It is used in various fields such as the spectral theory of random matrices, multidimensional bayesian analysis and more generally in statistics, where its importance stems from the fact that it is a multidimensional generalisation of the chi--squared distribution which appears naturally in the likelihood--test ratio. The Wishart distribution is also commonly used to analyse the problem of wave fading in wireless communication, which is of particular interest to us in view of the results presented in \S\ref{secsignalproc} below. For further details on this probability distribution, see, e.g., \cite{wishart}. We only mention here the few definitions and properties needed for our purpose.

Let $X$ be a random $n\times d$ matrix. Assume that the rows $\bm{x_i}$ ($1\le i\le n$) of $X$ are independent random vectors distributed according to a $d$--variate normal distribution $\mathcal{N}_d\left(\bm{0}, V \right)$ with zero mean and covariance matrix $V\in \mathcal{S}_d^{++}$. The Wishart distribution in dimension $d\ge 1$ with $n$ \emph{degrees of freedom} with respect to the \emph{scale matrix} $V$ is then the probability distribution of the matrix $\transp{X}X$. It is usually denoted by $\mathcal{W}_d(V, n)$. Whenever $n\ge d$, the matrix $\transp{X}X$ is invertible with probability one and the Wishart distribution admits a density function given by $$f_{\mathcal{W}_d(V, n)}(Q)\, =\, \frac{1}{2^{nd/2}\cdot \left|V\right|^{n/2}\cdot \Gamma_d\left(\frac{n}{2}\right)}\cdot \left|Q\right|^{(n-d-1)/2}\cdot \exp\left( -\frac{1}{2}\cdot\textrm{Tr}\left(V^{-1}Q\right)\right).$$ Here, $Q\in\mathcal{S}_d^{++}$, $\left|V\right|$ and $\left|Q\right|$ are shorthand notation for the determinant of $V$ and $Q$ respectively, $\textrm{Tr}(\, . \,)$ is the usual trace operator over the space of matrices and $$\Gamma_d\left(\frac{n}{2}\right)\, :=\, \pi^{d(d-1)/4}  \prod_{j=1}^{d}\Gamma\left( \frac{n}{2}+\frac{1-j}{2}\right)$$ is the multivariate Gamma function.

Let $\delta>0$. Denote by $m_{\mathcal{W}_d(V, n)}(\delta)$ the probability corresponding to the Wishart distribution defined as in~\eqref{defprobchol}. With the notation of Theorem~\ref{thmchol}, one has then the estimates 
\begin{equation}\label{estimwishart}
1-\int_{\sqrt{\delta}}^{\infty} g_{\mathcal{W}_d(V, n)}\; \le \; m_{\mathcal{W}_d(V, n)}(\delta) \; \le \; 1-\int_{I_d(\delta)} G_{\mathcal{W}_d(V, n)},
\end{equation} 
where the function $G_{\mathcal{W}_d(V, n)}$ is explicitly given for any $\bm{\beta}\in (\R_{>0})^d$ by $$G_{\mathcal{W}_d(V, n)}(\bm{\beta})\, = \, \frac{\prod_{i=1}^{d}\beta_i^{n-i}}{2^{d(n/2-1)}\cdot \left|V \right|^{n/2}\cdot \Gamma_d\left(\frac{n}{2}\right)}\cdot \int_{\R^p} \exp\left( -\frac{1}{2}\cdot\textrm{Tr}\left(V^{-1}\cdot\varphi_{chol}\left(\bm{\beta, \bm{u}} \right)\right)\right)\textrm{d}\lambda_p(\bm{u})$$ and where the function $g_{\mathcal{W}_d(V, n)}$ is defined as in~\eqref{defgchol}.

For the sake of concreteness, assume from now on that $$n=d=2 \quad \mbox{ and }\quad V=I_2.$$ Then, $$\varphi_{chol}\; : \; \begin{pmatrix} \beta_1 & u \\ 0 & \beta_2 \end{pmatrix}\, \mapsto \, \begin{pmatrix} \beta_1^2 & u\beta_1 \\ u\beta_1 & u^2+\beta_2^2 \end{pmatrix}$$ and, after calculations, $$G_{\mathcal{W}_2(I_2, 2)}(\beta_1, \beta_2)\, = \, \sqrt{\frac{2}{\pi}}\cdot\beta_1\cdot\exp\left(-\frac{1}{2}\left(\beta_1^2+\beta_2^2 \right) \right)$$ and $$g_{\mathcal{W}_2(I_2, 2)}\left( \beta_1\right)\, = \, \beta_1\cdot  \exp\left(-\frac{1}{2} \beta_1^2\right).$$ Inequalities~\eqref{estimwishart} now read~: 
\begin{align*}
J_1(\delta)&\,  :=\, 1-\int_{\sqrt{\delta}}^{+\infty} \beta_1\cdot  \exp\left(-\frac{1}{2} \beta_1^2\right) \cdot\textrm{d}\beta_1\,  \le \, m_{\mathcal{W}_2(I_2, 2)}(\delta) \\ 
& \le \, 1-\sqrt{\frac{2}{\pi}}\cdot \left( \int_{\sqrt{\delta}}^{+\infty}\beta_1\cdot\exp\left(-\frac{1}{2}\beta_1^2 \right)\cdot \textrm{d}\beta_1\right)\cdot \left( \int_{\sqrt{\delta}}^{+\infty}\exp\left(-\frac{1}{2}\beta_2^2 \right)\cdot \textrm{d}\beta_2\right)\, :=\, J_2(\delta).
\end{align*}

Some values taken by the functions $J_1$ and $J_2$ are represented in the following table~:

\begin{center}
\begin{tabular}{||c||c|c|c|c||}
\hline
\hline
$\delta$ & $0.2$ & $0.1$ & $0.01$ & $0.001$\\
\hline
\hline
$J_1(\delta)$ & 0.095 & 0.049 & $4.99\cdot 10^{-3}$ & $5.0\cdot 10^{-4}$\\ 
\hline
$J_2(\delta)$ & $0.41$ & $0.28$ & $8.42\cdot 10^{-2}$ & $2.6\cdot 10^{-2}$\\ 
\hline
\hline
\end{tabular}
\end{center}

If the space of two dimensional positive definite matrices is equipped with the probability distribution $\mathcal{W}_2(I_2, 2)$, the numerical values above imply for instance that at most $8.42\%$ of these matrices admit a minimum over $\Z^2\backslash\{ \bm{0}\}$ less than 0.01. Conversely, such a minimum is bigger than 0.2 for at least $9.5\%$ of these matrices.

The remainder of this section is devoted to the proof of Theorem~\ref{thmchol}.

\subsection{Proof of Theorem~\ref{thmchol}}\label{pruevcol} We first prove two preliminary lemmata. The first one is presented in a context slightly more general than the one imposed by Theorem~\ref{thmchol}~: this more general statement will be needed in \S\ref{secsignalproc} below. It involves the set
\begin{equation}\label{defMstargammacd}
\mathcal{M}_d^*\left(\gamma, c \right)\, :=\, \left\{H\in\mathcal{T}^{++}_d\; : \; \det\left( \gamma I_d+\transp{H}\cdot H\right)=c \right\}.
\end{equation}
Here, $I_d$ is the identity matrix in dimension $d$ and $\gamma$ and $c$ are non--negative real numbers. It is easily seen (with the help of the spectral decomposition for instance) that the set $\mathcal{M}_d^*\left(\gamma, c \right)$ is non--empty if, and only if, $c>\gamma^d$.

\begin{lem}\label{lemjac}
The map $\varphi_{chol}$ as defined in~\eqref{cholmap} is a $\mathcal{C}^1$--diffeomorphism with Jacobian determinant 
\begin{equation}\label{defjacchol}
\emph{\textrm{Jac}}_L\left(\varphi_{chol} \right)\, = \, 2^d\cdot \prod_{i=1}^{d}l_{ii}^{d-i+1}
\end{equation} for any $L\in\mathcal{T}^{++}_d$ with diagonal entries $(l_{11}, \dots, l_{dd})$. 

Also, assuming $c>\gamma^d$, the map 
\begin{equation*}
\Psi_{(\gamma, c)}^{(d)}~: H\,\in\, \mathcal{M}_d^*\left(\gamma, c \right)\: \mapsto \: c^{-1/d}\cdot \left(\gamma I_d+\transp{H}\cdot H\right)\,\in\, \Sigma_d^{++}
\end{equation*} is a $\mathcal{C}^1$--diffeomorphism between $\mathcal{M}_d^*\left(\gamma, c \right)$ and its image with Jacobian determinant 
\begin{equation}\label{defjaccholmod}
\emph{\textrm{Jac}}_H\left(\Psi_{(\gamma, c)}^{(d)} \right)\, = \, 2^{d-1}\cdot c^{-(d-1)(d+2)/(2d)}\cdot \prod_{i=1}^{d-1}h_{ii}^{d-i+1}
\end{equation}
for any $H\in\mathcal{M}_d^*\left(\gamma, c \right)$ with diagonal entries $(h_{11}, \dots, h_{dd})$. 
\end{lem}

\begin{proof}
Only equation~\eqref{defjaccholmod} will be established hereafter as equation~\eqref{defjacchol} can be deduced (in an easier way) from the argument presented below.

We first seek to determine a system of independent coordinates in $\mathcal{M}_d^*\left(\gamma, c \right)$ and in its image $\Psi_{(\gamma, c)}^{(d)}\left( \mathcal{M}_d^*\left(\gamma, c \right)\right)$. To this end, given $c>\gamma^d$, define the auxiliary polynomial map $$\widetilde{\Psi}^{(d)}_{\gamma}~: H\,\in \,\mathcal{T}^{++}_d\: \mapsto \: \det\left( \gamma I_d +\transp{H}\cdot H\right)$$ is such a way that $\mathcal{M}_d^*\left(\gamma, c \right) = \left( \widetilde{\Psi}^{(d)}_{\gamma}\right)^{-1}\left( \left\{ c\right\}\right)$.  Since the differential of the determinant map at a square matrix $A$ is the map $X\mapsto \textrm{Tr}\left(\transp{\textrm{com}(A)}\cdot X \right)$ (where $\textrm{com}(A)$ is the comatrix of $A$), an elementary calculation shows that, at any $H\in \mathcal{M}_d^*\left(\gamma, c \right)$, the differential $\textrm{d}_H \widetilde{\Psi}^{(d)}_{\gamma}$ of $\widetilde{\Psi}^{(d)}_{\gamma}$ is the linear map $$\textrm{d}_H\widetilde{\Psi}^{(d)}_{\gamma}~: X\in\mathcal{T}^{++}_d\, \mapsto \, 2\cdot \textrm{Tr}\left[c\cdot\left( \gamma I_d+\transp{H}\cdot H\right)^{-1}\cdot \transp{H}X \right].$$ This map has clearly rank one. From the Regular Value Theorem (see~\cite[Lemma 1 p.11]{zbMATH01950480}), the fibre $\mathcal{M}_d^*\left(\gamma, c \right)$ is therefore a manifold of dimension $\dim\mathcal{T}^{++}_d -1 = (d-1)(d+2)/2.$  

If $H=(h_{ij})_{1\le i\le j \le d}\in \mathcal{M}_d^*\left(\gamma, c \right)$, choose for a system of coordinates in $\mathcal{M}_d^*\left(\gamma, c \right)$ the $(d-1)(d+2)/2$ variables $\widetilde{h}:=(h_{ij})_{1\le i\le j \le d-1}$ (i.e.~excluding $h_{dd}$). Let $\Sigma:=\left(\sigma_{ij}\right)_{1\le i, j\le d}$ lie in the image of $\mathcal{M}_d^*\left(\gamma, c \right)$ by $\Psi_{(\gamma, c)}^{(d)}$. Let $\widetilde{\sigma}:=\left(\sigma_{ij} \right)_{1\le i\le j \le d-1}$ (this is the upper triangular part of $\Sigma$ excluding the term $\sigma_{dd}$).
In order to show that $\widetilde{\sigma}$ is a system of $(d-1)(d+2)/2$ independent coordinates parametrised by $\widetilde{h}$, express $\Sigma$ as $\Sigma = c^{-1/d}\cdot \left(\gamma I_d + \transp{H}\cdot H \right)$ for some $H \in \mathcal{M}_d^*\left(\gamma, c \right)$. Note then that when the elements of $\widetilde{\sigma}$ are listed row by row, each new entry 
\begin{equation}\label{defsigij}
\sigma_{ij}=c^{-1/d}\cdot\left(\gamma\delta_{ij} + \sum_{k=1}^{i} h_{ki}h_{kj}\right)
\end{equation} 
($1\le i\le j \le d-1$) depends on an entry of $H$ which has not appeared previously. However, $\sigma_{dd}=c^{-1/d}\cdot\left(\gamma + h_{dd}^2+ \sum_{k=1}^{d}h_{kd}^2 \right)$ can be expressed as a function of $\widetilde{h}$ and $h_{dd}$. For example, when $d=3$, $$\Sigma\, = \, c^{-1/d}\gamma I_d + c^{-1/d}\cdot \begin{pmatrix} h_{11}^2 & h_{11}h_{12} & h_{11}h_{13} \\ & h_{12}^2+h_{22}^2 & h_{12}h_{13}+h_{22}h_{23} \\ \ast & & h_{33}^2+h_{13}^2+h_{23}^2 \end{pmatrix}.$$ This legitimates $\widetilde{h}$ and $\widetilde{\sigma}$ as systems of coordinates respectively for $\mathcal{M}_d^*\left(\gamma, c \right)$ and for its image by $\Psi_{(\gamma, c)}^{(d)}$.

In order to compute the Jacobian determinant in~\eqref{defjaccholmod}, we now adapt the argument developed in~\cite[Chap.~7]{statmodel} to our purpose. Fix $H=(h_{ij})_{1\le i\le j \le d}\in \mathcal{M}_d^*\left(\gamma, c \right)$ and denote by $\left(\textrm{d}_{\Psi(H)}\,\sigma_{ij}\right)_{i, j}$ (resp.~by $\left(\textrm{d}_{H}\, h_{ij}\right)_{i, j}$) the canonical basis of the tangent space to $\Psi_{(\gamma, c)}^{(d)}\left( \mathcal{M}_d^*\left(\gamma, c \right)\right)$ at $\Psi_{(\gamma, c)}^{(d)}(H)$ with respect to the system of coordinates $\widetilde{\sigma}$ (resp.~of the tangent space to $\mathcal{M}_d^*\left(\gamma, c \right)$ at $H$ with respect to the system of coordinates $\widetilde{h}$). For the sake of simplicity of notation, set further $\textrm{d} \sigma_{ij}:= \textrm{d}_{\Psi(H)}\,\sigma_{ij}$ and $\textrm{d} h_{ij}:=\textrm{d}_{H} \,h_{ij}$. The latter tangent vectors then satisfy the property that for any $i,j $, 
\begin{equation}\label{tgtvectwedge}
\textrm{d} h_{ij}\wedge \textrm{d} h_{ij} = 0.
\end{equation}
Moreover, the change of coordinates induced by $\Psi_{(\gamma, c)}^{(d)}$ implies that $$\bigwedge_{1\le i, j \le d-1} \textrm{d} \sigma_{ij} \, = \, \textrm{Jac}_H\left(\Psi_{(\gamma, c)}^{(d)} \right) \cdot \bigwedge_{1\le i, j \le d-1} \textrm{d} h_{ij}$$ (see~\cite[Chap.~7]{statmodel} for details). In view of~\eqref{defsigij}, one has $$\textrm{d} \sigma_{ij}= c^{-1/d}\sum_{k=1}^{i}\left(h_{kj}\cdot \textrm{d}h_{ki}+h_{ki}\cdot\textrm{d}h_{kj}\right), $$ i.e.
\begin{align*}
c^{1/d}\textrm{d} \sigma_{11}\,&=\, 2h_{11}\cdot\textrm{d}h_{11},\\
c^{1/d}\textrm{d} \sigma_{12}\,&=\, h_{11}\cdot\textrm{d}h_{12}+\dots\, ,\qquad \dots \qquad\, , \\ 
c^{1/d}\textrm{d} \sigma_{1d}\,&=\, h_{11}\cdot\textrm{d}h_{1d}+\dots\, , \\
c^{1/d}\textrm{d} \sigma_{22}\,&=\, 2h_{22}\cdot\textrm{d}h_{22} + \dots\, , \, \qquad \dots \qquad \,  ,\\ 
c^{1/d}\textrm{d} \sigma_{2d}\,&=\, h_{22}\cdot\textrm{d}h_{2d}+\dots\, , \qquad  \dots\qquad \, , \\
\vdots \quad &\\
c^{1/d}\textrm{d} \sigma_{d-1, d-1}\,&=\, 2h_{d-1, d-1}\cdot\textrm{d}h_{d-1, d-1}\,  + \, \dots
\end{align*}
The point to write these expressions this way is that, in view of~\eqref{tgtvectwedge}, as soon as $\textrm{d}h_{ij}$ appears in one of the terms in $\textrm{d}\sigma_{ij}$, it may be ignored in all the others. All in all, this leads to $$c^{(d-1)(d+2)/(2d)}\bigwedge_{1\le i, j \le d-1} \textrm{d} \sigma_{ij} \, = \, \left(2^{d-1}\cdot \prod_{i=1}^{d-1}h_{ii}^{d-i+1} \right) \cdot \bigwedge_{1\le i, j \le d-1} \textrm{d} h_{ij},$$ which completes the proof of the lemma.
\end{proof}

The second lemma needed to prove Theorem~\ref{thmchol} is more elementary.

\begin{lem}\label{veccourt}
Let $L=(\bm{\beta}, \bm{\bm{u}})\in\mathcal{T}^{++}_d$ and $\eta>0$. Write $\bm{\beta}=\left(\beta_1, \dots, \beta_d \right)\in \left(\R_{>0} \right)^d$.
The following holds~: 
\begin{itemize}
\item if $\beta_i> \eta$ for all $i=1, \dots, d$, then 
\begin{equation}\label{interlattics}
L\cdot\Z^d\cap B_2\left(\bm{0}, \eta \right)\, = \, \left\{\bm{0} \right\};
\end{equation}

\item conversely, if $L\cdot\Z^d\cap B_2\left(\bm{0}, \eta \right)\, = \, \left\{\bm{0} \right\}$, then $\beta_1>\eta$.
\end{itemize}
\end{lem}

\begin{proof}
The second claim is immediate upon noticing that $\beta_1 = \left\|L\bm{e_1}\right\|_2$. Assume therefore that $\beta_i> \eta$ for all $i=1, \dots, d$ and note that  conclusion~\eqref{interlattics} is trivial when $d=1$. Let $d\ge 2$. Decompose the matrix $L:=L_d$ in the following way~: $$L_d\, = \, \begin{pmatrix} L_{d-1} & \bm{u_{d-1}}\\ \bm{0} & \beta_d\end{pmatrix}.$$ Here, $L_{d-1}\in\mathcal{T}^{++}_{d-1}$ and $\bm{u_{d-1}}\in\R^{d-1}.$ It is then readily seen that $$L\cdot\Z^d \, = \, \bigcup_{n\in\Z} A_{L_d}(n), \qquad \textrm{ where } \qquad A_{L_d}(n)\, = \, \begin{pmatrix} L_{d-1}\cdot\Z^{d-1} +n\bm{u_{d-1}}\\ n\beta_d \end{pmatrix}.$$ Proceeding by induction on $d\ge 2$, given $\bm{x}\in A_{L_d}(n)$, the inequality $\left\|\bm{x} \right\|_{\infty} > \eta$ follows by the induction hypothesis if $n=0$ and is otherwise a direct consequence of the fact that $\left\|\bm{x} \right\|_{\infty} \ge \beta_d>\eta$. This completes the proof of the lemma.
\end{proof}

\begin{proof}[Proof of Theorem~\ref{thmchol}]
Only the estimates~\eqref{cholestim2} will be established hereafter as inequalities~\eqref{cholestim1} follow  from the argument presented below in a similar way.

Let $\Sigma\in\Sigma^{++}_d$ decomposed in its Cholesky form as $\Sigma=\transp{L}L$, where $L=(\bm{\beta'}, \bm{u})\in\Theta^{++}_d$ with $\bm{\beta'}= (\beta'_1, \dots, \beta'_{d-1})\in\left( \R_{>0}\right)^{d-1}$ and $\bm{u}\in \R^p$. Set furthermore $$\beta'_d=\left(\prod_{k=1}^{d-1}\beta'_k\right)^{-1}.$$

It should be clear that, given $\delta>0$, $$\left(M_d(\Sigma)>\delta \right) \quad \iff \quad \left(L\cdot\Z^d\cap B_2\left(\bm{0}, \sqrt{\delta} \right)\, = \, \{\bm{0}\} \right).$$ From Lemma~\ref{veccourt}, if either statement in this equivalence holds, then $\beta'_1>\delta$. Conversely, it also follows from Lemma~\ref{veccourt} that if $\min_{1\le i\le d} \,\beta'_i >\sqrt{\delta}$, that is, if $\bm{\beta'}\in\Delta_{d-1}(\delta)$, then any of the statements in this equivalence holds. 

Since 
\begin{align*}
1-\widetilde{m}_{\widetilde{f}}(\delta)\, &= \, \int_{\Sigma_d^{++}}\widetilde{f}(\Sigma)\cdot \chi_{[M_d(\Sigma)>\delta]}\cdot \textrm{d}\Sigma\\
&= \, \int_{\Theta^{++}_d}\left(\widetilde{f}\circ\widetilde{\varphi}_{chol} \right)(L)\cdot\left|\textrm{Jac}_{L}\left(\widetilde{\varphi}_{chol} \right)\right|\cdot \chi_{[L\cdot\Z^d\cap B_2\left(\bm{0}, \sqrt{\delta} \right)\, = \, \{\bm{0}\}]}\cdot\textrm{d}L,
\end{align*}
one thus obtains the estimates
\begin{align*}
\int_{\Delta_{d-1}(\delta)}\textrm{d}\lambda_{d-1}(\bm{\beta'})\int_{\R^p}\left(\widetilde{f}\circ\widetilde{\varphi}_{chol} \right)&(\bm{\beta'}, \bm{u})\cdot \left|\textrm{Jac}_{(\bm{\beta'}, \bm{u})}\left(\widetilde{\varphi}_{chol} \right)\right|\cdot\textrm{d}\lambda_{p}(\bm{u})\\
\le \, 1- &\widetilde{m}_{\widetilde{f}}(\delta)\,\le\\
\int_{\sqrt{\delta}}^{+\infty}\textrm{d}\lambda(\beta'_1) \int_{\left(\R_{>0}\right)^{d-2}}\textrm{d}\lambda_{d-2}(\bm{\utilde{\beta'}})\int_{\R^p}\left(\widetilde{f}\circ\widetilde{\varphi}_{chol} \right)&(\beta'_1, \bm{\utilde{\beta}'}, \bm{u})\cdot \left|\textrm{Jac}_{(\beta'_1, \bm{\utilde{\beta}'}, \bm{u})}\left(\widetilde{\varphi}_{chol} \right)\right|\cdot\textrm{d}\lambda_{p}(\bm{u})
\end{align*}
(recall that $\bm{\beta'} = \left(\beta'_1, \bm{\utilde{\beta}'} \right)$).
The upper and lower bounds for $\widetilde{m}_{\widetilde{f}}(\delta)$ in~\eqref{cholestim2} now follow directly from Lemma~\ref{lemjac} (with $\gamma=0$ and $c=1$). Furthermore, to prove that these bounds always lie in the interval $[0, \, 1]$, it is enough to notice that, from the definitions of the functions $\widetilde{G}_{\widetilde{f}}$ and $\widetilde{g}_{\widetilde{f}}$, $$\int_{\left(\R_{>0}\right)^{d-1}}\widetilde{G}_{\widetilde{f}}\, = \, \int_{0}^{+\infty}\widetilde{g}_{\widetilde{f}}\, = \, \int_{\Sigma_d^{++}}\widetilde{f}(\Sigma)\cdot \textrm{d}\Sigma\, = \, 1.$$
\end{proof}

\section{Application to Signal Processing} \label{secsignalproc}

The initial motivation of this work was to address a fundamental problem that emerged very recently in Information Theory. The latter is related to a new model of communication channel (the so called \emph{Integer--Forcing Architecture}) which has been receiving considerable attention in the literature due to its expected high performance. The precise estimation of this performance involves the probability that a quadratic form admits a minimum over non--zero lattice points less than a given constant.

In what follows, we first present the very basic tools from Information Theory that will enable one to understand the importance and the position of the problem under con\-si\-de\-ra\-tion --- for a deeper introduction to the topic, see~\cite{fondelectr}, especially Chapter~5. The theory developed in the previous sections will then allow one to bound accurately the probability to estimate. 

\subsection{Position of the Problem}\label{positionpb}

Assume that two \emph{users} (or \emph{transmitters}) $S_1$ and $S_2$ want to \emph{transmit} messages (or \emph{signals}) $x_1$ (for $S_1$) and $x_2$ (for $S_2$) along a communication channel (e.g., a cable or a radio channel) simultaneously to two \emph{receivers} $R_1$ and $R_2$ (\footnote{This configuration, widely studied in Information Theory, is known as an ``X--Channel'' with a reference to the shape of Figure~\ref{figure1} below.}). Independently of the familiar concept of noise, the signal is distorted during transmission up to a certain degree of \emph{fading}. This may be due for instance to the distance between the users and the receivers or else to reflections on obstacles such as buildings in the path of the signals. This phenomenon is modelled by the so--called \emph{channel coefficients}. For the message sent by $S_i$ to $R_j$ ($i, j\in\{1, 2\}$) the corresponding channel coefficient is denoted by $h_{ij}$. Thus, in the simplest case of an additive channel, the message $y_i$ received by $R_i$ ($i\in\{1, 2\}$) is represented by the system of equations
\begin{equation} \label{systeqtransm}
\left\{
    \begin{array}{l}
        y_1   =  h_{11} x_1   +  h_{12} x_2 + z_1 \\
        y_2   =  h_{21} x_1   +  h_{22} x_2  + z_2,
    \end{array}
\right.
\end{equation}
where $z_1$ and $z_2$ are the noise --- see also the figure below. 

\paragraph{}
\setlength{\unitlength}{20mm}

\vspace{0.5cm}

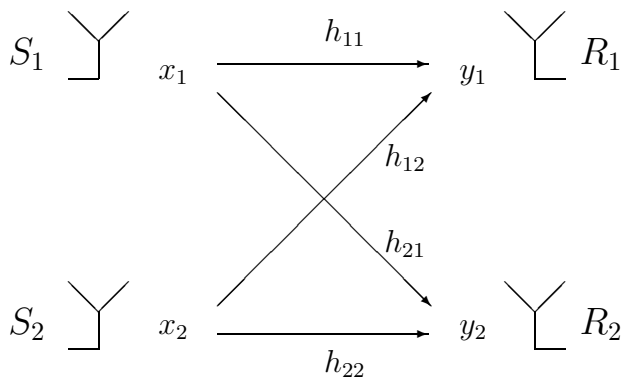
\begin{figure}[h!]\label{figure1}
\begin{center}
\begin{picture}(2, 2)(-0.75,-0.75)

\put(-1.5,0.8){\line(0,1){0.25}}
\put(-1.5,1.05){\line(-1,1){0.2}}
\put(-1.5,1.05){\line(1,1){0.2}}
\put(-1.5,0.8){\line(-1,0){0.2}}
\put(-1.1,0.8){$x_1$}

\put(1.4,0.8){\line(0,1){0.25}}
\put(1.4,1.05){\line(-1,1){0.2}}
\put(1.4,1.05){\line(1,1){0.2}}
\put(1.4,0.8){\line(1,0){0.2}}
\put(0.9,0.8){$y_1$}

\put(-1.5,-1){\line(0,1){0.25}}
\put(-1.5,-0.75){\line(-1,1){0.2}}
\put(-1.5,-0.75){\line(1,1){0.2}}
\put(-1.5,-1){\line(-1,0){0.2}}
\put(-1.1,-0.9){$x_2$}

\put(1.4,-1){\line(0,1){0.25}}
\put(1.4,-0.75){\line(-1,1){0.2}}
\put(1.4,-0.75){\line(1,1){0.2}}
\put(1.4,-1){\line(1,0){0.2}}
\put(0.9,-0.9){$y_2$}

\put(0,0){\vector(1,1){0.71}}
\put(0,0){\line(-1,-1){0.71}}
\put(0,0){\line(-1,1){0.71}}
\put(0,0){\vector(1,-1){0.71}}

\put(-0.71,0.9){\vector(1,-0){1.4}}
\put(-0.71,-0.9){\vector(1,-0){1.4}}

\put(-2.1,0.9){\Large{$S_1$}}
\put(-2.1,-0.9){\Large{$S_2$}}
\put(1.7,0.9){\Large{$R_1$}}
\put(1.7,-0.9){\Large{$R_2$}}

\put(0,1.05){\large{$h_{11}$}}
\put(0,-1.15){\large{$h_{22}$}}
\put(0.4,0.22){\large{$h_{12}$}}
\put(0.4,-0.35){\large{$h_{21}$}}
\end{picture}
\vspace{6mm}
\caption{Channel of communication corresponding to the configuration in~\eqref{systeqtransm}}
\end{center}
%\vspace{6mm}
\end{figure}

Matricially, the system of equations~\eqref{systeqtransm} reads
\begin{equation}\label{matricerepr}
\bm{y} \, = \, H\bm{x} +\bm{z}
\end{equation}
with 
\begin{equation*}
\bm{y} \, = \, \begin{pmatrix} y_1 \\ y_2 \end{pmatrix}, \quad   H\, = \, \begin{pmatrix} h_{11} & h_{12} \\ h_{21} & h_{22} \end{pmatrix}, \quad \bm{x} \, = \, \begin{pmatrix} x_1 \\ x_2 \end{pmatrix} \quad \textrm{ and }\quad \bm{z}\, =\, \begin{pmatrix} z_1 \\ z_2 \end{pmatrix}.
\end{equation*}

Of course, it is obvious to generalise this model to the case when there are $m\ge 1$ users and $n\ge 1$ receivers. Then, the matrix $H$ in~\eqref{matricerepr} is rectangular with dimensions $n\times m$, the vectors $\bm{y}$ and $\bm{z}$ are $n$--dimensional and the vector $\bm{x}$ is $m$--dimensional. From the receiver's point of view, it is natural to consider $\bm{x}$ and $\bm{z}$ as random vectors, in which case the entries of the noise vector $\bm{z}$ are often taken as independent with Gaussian distribution with zero mean and unit variance. As for the input $\bm{x}$, it satisfies a power constraint of the form 
\begin{equation}\label{defpowerconstraint}
\mathbb{E}\left(\transp{\bm{x}}\cdot\bm{x}\right) \le m\cdot \textrm{SNR},
\end{equation}
 where $\mathbb{E}(\, . \,)$ denotes the expectation and where $\textrm{SNR}$ stands for the \emph{Signal--to--Noise Ratio}, a fundamental strictly positive quantity which will be discussed later. In the standard case when each entry of $\bm{x}$ is a sum of binary digits (\emph{bits}), the power constraint~\eqref{defpowerconstraint} reflects the fact that the number of bits that can be sent through the channel is limited by some of its physical properties. 
 
It is important to point out here that the seemingly simple model with two users and two receivers exposed above underpins some of the most fundamental features of the more general model with $m$ users and $n$ receivers. Thus, some channel architectures with $m=2$ users and $n=2$ receivers have been at the heart of deep theoretical problems in Information Theory --- see, e.g., \cite[\S 5.4.3]{fondelectr}.

The most basic problem when considering a channel of communication is to determine whether the received information is reliable; that is, to what extent the noise affects the quality of the signal. In order to make the probability error small, an obvious guess is that one has to reduce the rate of new data sent by the users (for instance, by repeating each string of message several times). In 1948, Shannon proved that this intuition is surprisingly incorrect~: it is actually possible to exchange information at a \emph{strictly positive} data rate keeping at the same time the error probability as small as desired. There is nevertheless a maximal rate, the \emph{capacity} of the channel, above which this cannot be done any more. The latter quantity is usually expressed in bits.

As the proof of the result established by Shannon is non--effective (i.e.~it does not provide a way to code the information in order to approach the capacity), from an engineering standpoint, the problem to determine the capacity of a channel and then to provide a way to get as close as possible to this capacity remains open. 

There is no single expression for the capacity of a channel; rather, it depends on its intrinsic architecture. It nevertheless always involves the Signal--to--Noise Ratio (SNR). This quantity, often expressed in decibels, compares the level of a desired signal to the level of background noise~: the bigger this ratio, the better the quality of the signal. For the model represented by the equations in~\eqref{matricerepr} and~\eqref{defpowerconstraint} (with any $m, n \ge 1$), it is shown in~\cite{oruri} that the capacity $C$ can be expressed as 
\begin{equation}\label{defcapa}
C\,= \, \log \det\left(I_m + \textrm{SNR}\cdot\transp{H}\cdot H \right).
\end{equation}

Note also the following important point~:  the performances of a channel depend heavily on whether or not the transmitter knows the channel coefficients matrix $H$. Indeed, if such information is available, they can for instance allocate more power to the stronger antennas to minimise the effect of fading. In most cases however (for instance in wireless systems), this information is not known to the transmitter, in which case a reasonable strategy is to allocate equal power to each of the antennas. In the latter configuration, the capacity of the channel is rather referred to as the \emph{mutual information}.

\subsection{Channels with Integer--Forcing Receiver Architecture} Recently, an important breakthrough has been achieved in Information Theory. Indeed, Zhan \& \emph{alii} introduced in~\cite{colodelectriciens} a new architecture of channel, the so--called \emph{Integer--Forcing Receiver Architecture}, which has been receiving considerable attention in the literature (see~\cite{oruri} and the references therein for further details). It is not our goal to describe the channel precisely~: if interested, the reader is referred to~\cite{colodelectriciens}. Here is however the main ingredient from which follow all the properties of this new model~: in a standard communication channel, the receivers obtain the message $\bm{x}$ sent to them by first eliminating interferences  from the vector $\bm{y}$ (especially the noise $\bm{z}$) and then by decoding each individual data stream (i.e.~each component of the vector $\bm{y}$). The idea introduced by Zhan \& \emph{alii} is, first to decode integer linear combinations of data stream and, then, to eliminate the interference. 

The near optimality of this strategy has been verified by extensive \emph{ad hoc} calculations (see~\cite[\S I.A.]{oruri} for details). As for a theoretical proof of this fact, this task has been started in~\cite{oruri} in the following general set--up which, as explained in the paper, appears in several important communication scenarios.

Assume that each transmitter wishes to send the same message to all the receivers (this is for instance the case for TV broadcast). They all are aware of the characteristics of the channel, namely its SNR coefficient and also the mutual information $C_0$. However, they ignore the actual channel matrix $H$ modelling the transmission as in~\eqref{matricerepr}. Without any more information and in view of~\eqref{defcapa}, this matrix $H$ is considered as being randomly and ``uniformly'' chosen\footnote{As will be shown later, this concept of uniformity, understood here intuitively, needs to be clarified.} from the set 
\begin{equation}\label{defHbig}
\mathcal{H}_{m,n}\left(C_0, \textrm{SNR} \right)\, := \, \left\{H\in\R^{n\times m} \; : \; \log \det\left(I_m + \textrm{SNR}\cdot\transp{H}\cdot H \right) \, = \, C_0\right\}.
\end{equation}

It is proved in~\cite{oruri} that the performance of the channel under consideration after applying the integer--forcing technique is actually determined by the so--called \emph{Effective Signal--to--Noise Ratio} $\textrm{SNR}_{\textrm{eff}}$. We shall not be concerned with the actual definition of this quantity, which is rather technical --- for details, see~\cite[\S II.B.]{oruri}. The crucial point formulated with our notation is the following estimates satisfied by the  $\textrm{SNR}_{\textrm{eff}}$ coefficient (see~\cite[Theorem 2]{oruri} for a proof)~: 
\begin{equation}\label{inegsnreff}
\frac{1}{4m^2}\cdot M_m\left(I_m+\textrm{SNR}\cdot \transp{H}\cdot H\right)\, < \, \textrm{SNR}_{\textrm{eff}}\, \le \,  M_m\left(I_m+\textrm{SNR}\cdot \transp{H}\cdot H\right).
\end{equation}
 
For the quality of communication to be best possible, one wishes to obtain a $\textrm{SNR}_{\textrm{eff}}$ coefficient as large as possible. Inequalities~\eqref{inegsnreff} show that the order of magnitude of this coefficient is dictated by the minimum of the positive definite quadratic form $I_m+\textrm{SNR}\cdot \transp{H}\cdot H$ over non--zero elements of $\Z^m$. In view of the probabilistic model developed so far, the main problem which emerges from this theory can be formulated as follows~:

\begin{pb}[Main Problem of Application]\label{pb3}
Assume that the channel matrix $H$ is chosen randomly and ``uniformly'' from the set~\eqref{defHbig}. Let $\kappa\in (0,  1)$.

Find the best possible value of $s\ge 0$ such that the event $\textrm{SNR}_{\textrm{eff}}\ge s$ is realised with probability greater than $\kappa$; equivalently, determine the cumulative distribution function of the quantity $\textrm{SNR}_{\textrm{eff}}$ seen as a random variable.
\end{pb}

It is worth noting that the techniques developed here in order to tackle this problem can also be used to solve other questions appearing in the literature dealing with the Integer--Forcing Architecture. An example of such questions is the estimate of the probability that the so--called effective noise variance as defined in~\cite[\S IV.E.]{colodelectriciens} should be less than a given constant. Another more general example is the estimate of the so--called probability of outage of some channels --- see~\cite{fondelectr, colodelectriciens}. In all cases, the main ingredient is Theorem~\ref{thmchol} (more precisely, the upper bounds appearing therein). Also, it must be pointed out that the manifold~\eqref{defHbig} is ubiquitous in the literature related to Signal Processing. Some of its topological properties playing a crucial role in the study of the performance of various channels are established in \S\ref{proptopovariete} below.

\subsection{Formalisation of the Concept of a ``Uniformly'' Distributed Measure on the Set $\mathcal{H}_{m,n}\left(C_0, \textrm{SNR} \right)$}\label{proptopovariete} For convenience, set from now on 
\begin{equation}\label{chgmtvar}
\gamma:= (\textrm{SNR})^{-1}\quad \textrm{ and } \quad c_0:=\gamma^{m} e^{C_0}
\end{equation} in such a way that $$\mathcal{H}_{m,n}\left(C_0, \textrm{SNR} \right)\, = \, \left\{H\in\R^{n\times m} \; : \; \det\left(\gamma I_m + \transp{H}\cdot H \right) \, = \, c_0\right\}.$$  For the sake of simplicity of notation, the dependency of the various quantities on $\gamma$ and $c_0$ will not be marked hereafter. The reader should however keep in mind that almost all the constants, sets and functions introduced hereafter depend on these two parameters.

A crucial remark is that Sylvester's determinant identity immediately implies that $$\det\left(\gamma I_m + \transp{H}\cdot H \right) \, = \, \det\left(\gamma I_n + H\cdot\transp{H} \right).$$ Therefore, even if it means working throughout with $\transp{H}$ instead of $H$ to obtain the analogues in the case $n\ge m$ of the results stated below, it may be assumed \emph{without loss of generality} that 
\begin{equation}\label{wlgd=n}
d:=\min\left\{m, n \right\} = m.
\end{equation}

In order to address Problem~\ref{pb3} as stated above, one needs first to formalise the idea of a ``uniform'' measure on the set $\mathcal{H}_{m,n}\left(C_0, \textrm{SNR} \right)$. If one understands this concept in the usual mathematical meaning of a Borelian measure in a complete metric space such that the measure of a ball depends only on its radius but not on the position of its center, this is problematic. Indeed, as shown in Lemma~\ref{lemoutilpropH} below, the set $\mathcal{H}_{m,n}\left(C_0, \textrm{SNR} \right)$ is compact. Now, it is proved in~\cite[Proposition~1.7]{unifsphere} that a bounded subset of an Euclidean space carries a uniform measure only if it is contained in a sphere. It is not hard to see that this never happens for the set $\mathcal{H}_{m,n}\left(C_0, \textrm{SNR} \right)$ as soon as $d\ge 2$. In view of this and in order to render this idea of uniform distribution in a different way, we first establish some properties of the set $\mathcal{H}_{m,n}\left(C_0, \textrm{SNR} \right)$.

Given an integer $k\in \llbracket 0, d \rrbracket$, let $\mathcal{R}_{m,n}^{(k)}$ be the subset of $\mathcal{H}_{m,n}\left(C_0, \textrm{SNR} \right)$ consisting of all those matrices with rank $k$~: $$\mathcal{R}_{m,n}^{(k)} \, :=\, \left\{H\in \mathcal{H}_{m,n}\left(C_0, \textrm{SNR} \right)\; : \; \textrm{rank}(H)=k \right\}.$$ Note that any of the sets $\mathcal{R}_{m,n}^{(k)}$ is invariant under a map of the form $H\mapsto U\cdot H$, where $U\in\mathcal{O}_n$ is an orthogonal transformation. This legitimate the focus on a fundamental domain for the left action of $\mathcal{O}_n$ on $\mathcal{R}_{m,n}^{(k)} $. As shown in Lemma~\ref{lemoutilpropH} below, such a fundamental domain is naturally be related to the set $$\mathcal{M}^{(k)}_{d}\, :=\, \left\{T\in\mathcal{T}_d^+\; : \; \textrm{rank}(T)=k \quad \textrm{and} \quad \det\left(\gamma I_d+\transp{T}\cdot T\right)=c_0\right\},$$ where $\mathcal{T}_d^+$ is the set of all those upper triangular $d$--dimensional square matrices with non--negative diagonal entries. Note that when $k=d$, the set $\mathcal{M}^{(k)}_{d}$ coincides with the set $\mathcal{M}_d^*\left(\gamma, c_0 \right)$ defined in~\eqref{defMstargammacd}. In what follows, we will adopt the simpler notation $$\mathcal{M}_d^*\, := \,\mathcal{M}_d^*\left(\gamma, c_0 \right).$$

It is not hard to see that a necessary and sufficient condition for the subset $\mathcal{M}^*_d$ to be non--empty is that 
\begin{equation}\label{inegc0gamma}
c_0 > \gamma^d.
\end{equation}
In this case, the zero matrix cannot belong to the set 
\begin{equation}\label{defmtildedgamc}
\widetilde{\mathcal{M}}_d\, :=\, \bigcup_{k=0}^{d} \mathcal{M}_d^{(k)}\, = \, \left\{T\in\mathcal{T}_d^+\; : \;  \det\left(\gamma I_d+\transp{T}\cdot T\right)=c_0\right\}
\end{equation} 
(if $c_0 = \gamma^d$, the latter set only contains the zero matrix and if $c_0 < \gamma^d$, it is empty --- see \S\ref{pruevcol} or the proof of Lemma~\ref{lemoutilpropH} for details). The relation~\eqref{inegc0gamma} will be assumed to hold throughout. 

\begin{lem}\label{lemoutilpropH}
The following two points hold~:

\begin{itemize}
\item The set $\mathcal{H}_{m,n}\left(C_0, \textrm{SNR} \right)$ is compact.

\item Given an integer $k\in \llbracket 0, d \rrbracket$, a fundamental domain for the left action of the orthogonal group  $\mathcal{O}_n$ on $\mathcal{R}_{m,n}^{(k)} $ can naturally be identified with a subset of $\mathcal{M}^{(k)}_{d}$. Furthermore, when $k=d$,  a fundamental domain for the left action of the orthogonal group  $\mathcal{O}_n$ on $\mathcal{R}_{m,n}^{(d)}$ can naturally be identified with the set $\mathcal{M}_d^*$ itself.
\end{itemize}
\end{lem}

\begin{proof}
The second point is a direct consequence of the $QR$ decomposition~: any matrix $H\in \mathcal{R}_{m,n}^{(k)}$ can be decomposed as $H=QR$, where $Q\in\mathcal{O}_n$ and where the matrix $R$ has rank $k$ and is of the form $$R=\begin{pmatrix} T\\ \bm{0}\end{pmatrix}$$ with $T\in \mathcal{T}_d^+$. Furthermore, this decomposition is unique when $R$ has full rank.

As for the first point, note that the set $\mathcal{H}_{m,n}\left(C_0, \textrm{SNR} \right)$ is clearly closed. To show that it is also bounded, we will adopt the following notation~: given a $n\times m$ rectangular matrix $M$, $\left\|M \right\|_{\infty}$ will denote the sup--norm of the vector in $\R^{nm}$ determined by its entries. Also, $\vvvert M \vvvert_2$ (resp.~$\vvvert M \vvvert_\infty$) will stand for the operator norm of $M$ induced by the Euclidean norms (resp.~the sup--norms). Given two positive real numbers $a$ and $b$, the Vinogradov symbol $a\ll b$ will as usual indicate the existence of a positive constant $c>0$ such that $a\le cb$.

Let then $H\in \mathcal{H}_{m,n}\left(C_0, \textrm{SNR} \right)$. By looking at the diagonal elements in $\transp{H}\cdot H$, it is plain that $$\left\| H\right\|_\infty\, \le \,\sqrt{\left\| \transp{H}\cdot H\right\|_{\infty}}.$$ Let $\transp{H}\cdot H \, = \, \transp{P}\cdot D\cdot P$ be the spectral decomposition of the positive matrix $\transp{H}\cdot H$, where $P\in\mathcal{O}_m$ and where $D$ is a diagonal matrix with entries $\lambda_1, \dots, \lambda_m\ge 0$. From the equivalence of norms in finite dimension and from the fact that $\vvvert P \vvvert_2 =1$, one thus obtains~: 
\begin{align*}
\left\| \transp{H}\cdot H\right\|_{\infty}\, &\ll \, \vvvert \transp{H}\cdot H\vvvert_{2}\, = \, \vvvert \transp{P} D  P\vvvert_{2} \\
&\le \, \vvvert \transp{P}\vvvert_{2}  \vvvert D\vvvert_{2} \vvvert P\vvvert_{2}\\
&= \, \vvvert D\vvvert_{2}\\
&\ll \, \left\| D\right\|_\infty \, := \, \max\, \textrm{Spect}\left(\transp{H}\cdot H \right),
\end{align*}
where $\textrm{Spect}\left(\transp{H}\cdot H \right)$ denotes the spectrum of the matrix $\transp{H}\cdot H$. From the definition of the set  $\mathcal{H}_{m,n}\left(C_0, \textrm{SNR} \right)$, one has furthermore that $$c_0\, = \, \det\left(\gamma I_m + \transp{H}\cdot H \right) \, = \, \det\left(\gamma I_m +D \right)\, = \, \prod_{i=1}^m\left(\gamma+\lambda_i \right).$$ Since $\lambda_i\ge 0$ for all $i=1, \dots, m$, this implies that $\textrm{Spect}\left(\transp{H}\cdot H \right)\subset \left[0, \, \gamma\left( c_0\gamma^{-m}-1\right)\right]$ (which set is empty if $c_0<\gamma^m$). This completes the proof.
\end{proof}

\begin{remark} We would like to point out here that the first point in Lemma~\ref{lemoutilpropH} rules out an assumption often made in the literature related to Information Theory (see, among many other examples, \cite[Problem 13.12]{zhangwire}); namely, the coefficients of a matrix $H$ lying in $\mathcal{H}_{m,n}\left(C_0, \textrm{SNR} \right)$ cannot have a Gaussian distribution.
\end{remark}

\begin{remark}\label{rem2} A much more involved argument presented in the proof of Lemma~\ref{lemdetrnirer} below implies that the Euclidean norm of a matrix lying in the set $\widetilde{\mathcal{M}}_d$ and viewed as a vector in $\R^{d(d+1)/2}$ is at most $\sqrt{(c_0-\gamma^d)/\gamma^{d-1}}$ and at least $\sqrt{c_0^{1/d}-\gamma}$ --- see the end of \S\ref{pruevelem6} for details. From the $QR$ decomposition, this also holds for a matrix lying in $\mathcal{H}_{m,n}\left(C_0, \textrm{SNR} \right)$.
\end{remark}

If one understands the concept of a ``uniform'' measure as a measure ``evenly'' distributed  (in some intuitive sense), in view of the invariance of the set $\mathcal{H}_{m,n}\left(C_0, \textrm{SNR} \right)$ under the left action of the orthogonal group, it is natural to define such a measure from a fundamental domain of $\mathcal{H}_{m,n}\left(C_0, \textrm{SNR} \right)$  for this action. Thus, if one is able to equip the set $\widetilde{\mathcal{M}}_d$ as defined in~\eqref{defmtildedgamc} with a ``uniform'' probability measure $\widetilde{\nu}_d$ which satisfies furthermore the property that 
\begin{equation}\label{propomeasurg}
\widetilde{\nu}_d\left(\mathcal{M}^*_d \right)=1
\end{equation} 
(that is, the measure $\widetilde{\nu}_d$ is only supported on those matrices of full rank), then, in view of Lemma~\ref{lemoutilpropH},  $\widetilde{\nu}_d$ would be a relevant candidate for our purpose\footnote{It must be pointed out here that, from an engineering standpoint, it is often assumed that the channel matrix has full rank not to have to deal with redundant information. Lemma~\ref{mesurergplein} below shows that we will not have to make such an assumption here.}. 

A natural choice for $\widetilde{\nu}_d$ is a measure which takes into account the geometry of the manifold $\widetilde{\mathcal{M}}_d$. Setting $$p'=\frac{d(d+1)}{2}-1,$$ this leads one to define $\widetilde{\nu}_d$ from the infinitesimal volume element $\textrm{d}\, \textrm{vol}_{p'}(T)$ on the hypersurface $\widetilde{\mathcal{M}}_d\subset \R^{p'+1}$. More precisely, for any measurable subset $\mathcal{B}\subset \widetilde{\mathcal{M}}_d$, 
\begin{equation}\label{defnutilde}
\widetilde{\nu}_d\left( \mathcal{B}\right)\, := \, \frac{\int_{\mathcal{B}}\textrm{d}\, \textrm{vol}_{p'}(T)}{\int_{\widetilde{\mathcal{M}}_d}\textrm{d}\, \textrm{vol}_{p'}(T)}\cdotp
\end{equation}
Note that this is a well--defined probability measure as $\widetilde{\mathcal{M}}_d$ is compact.

Let $$f~: T \in\mathcal{T}_d^+ \; \mapsto \; c_0^{-1/d}\cdot \left(\gamma I_d+ \transp{T}\cdot T \right)$$ and 
\begin{equation}\label{defg}
g\, := \, \det\circ f
\end{equation} 
in such a way that $$\widetilde{\mathcal{M}}_d\, = \, g^{-1}\left(\left\{ 1\right\} \right).$$ Given $T=\left(t_{ij}\right)_{1\le i, j, \le d} \in\mathcal{T}_d^+ $ and given indices $i$ and $j$ such that $1\le i\le j \le d$, set $$\partial_{ij}\,:=\,\frac{\partial}{\partial t_{ij}}$$ and define furthermore the charts 
\begin{equation}\label{defcartesBij}
\mathcal{B}_{ij}\, :=\, \left\{T \in\mathcal{T}_d^+ \; : \;  \left(\partial_{ij} g\right)(T)\neq 0\right\}.
\end{equation}
The relevance of this definition follows from this lemma~:

\begin{lem}\label{lemdetrnirer}
Assume~\eqref{inegc0gamma}. Then~:
\begin{itemize}
\item The gradient $\nabla g$ of $g$ never vanishes on $\widetilde{\mathcal{M}}_d$. In other words, $$\widetilde{\mathcal{M}}_d\, = \, \bigcup_{1\le i\le j \le d} \left(\mathcal{B}_{ij}\cap  \widetilde{\mathcal{M}}_d\right).$$

\item On each of the charts $\mathcal{B}_{ij}$, the volume element $\emph{\textrm{d}}\, \textrm{vol}_{p'}(T)$ can be expressed as follows~: 
\begin{equation}\label{elmtvol}
\emph{\textrm{d}}\, \textrm{vol}_{p'}(T)\, = \, \left(\frac{\left\|\nabla g \right\|_2}{\left| \partial_{ij}g\right|} \right) (T)\cdot \emph{\textrm{d}}t_{11}\dots \widehat{\emph{\textrm{d}}t_{ij}}\dots \emph{\textrm{d}}t_{dd}
\end{equation}
(as usual, the hat means that the corresponding index is removed from the list).

\item The subset of matrices of full rank in $\widetilde{\mathcal{M}}_d$ is contained in $\mathcal{B}_{dd}$~: $$\mathcal{M}^*_d\, \subset \, \mathcal{B}_{dd}.$$
\end{itemize}
\end{lem}

With the help of this lemma, one can now prove that the measure $\widetilde{\nu}_d$ defined in~\eqref{defnutilde} satisfies~\eqref{propomeasurg}~:

\begin{lem}\label{mesurergplein}
Let $k\in\llbracket 0, d-1 \rrbracket$. Then, under~\eqref{inegc0gamma}, $$\widetilde{\nu}_d\left(\mathcal{M}_d^{(k)}\right)\, = \, 0$$
\end{lem}

\begin{proof}
It follows from Lemma~\ref{lemdetrnirer} that $\widetilde{\mathcal{M}}_d$ can be covered by a finite number of subsets $\left(\mathcal{B}'_{ij}\right)_{1\le i\le j \le d}$ such that, within each $\mathcal{B}'_{ij}$, the function $\partial_{ij}g$ never vanishes. Also, within each $\mathcal{B}'_{ij}$, the measure determined by the volume element $\textrm{d}\, \textrm{vol}_{p'}(T)$ is absolutely continuous with respect to the $p'$--dimensional Lebesgue measure $\lambda_{p'}$. In order to prove the lemma, it is therefore enough to establish that for all $0\le k\le d-1$ and all $1\le i \le j \le d$, 
\begin{equation}\label{objectiflebzero}
\lambda_{p'}\left(\mathcal{M}_d^{(k)}\cap \mathcal{B}'_{ij} \right) \, = \, 0.
\end{equation} 
To this end, note that $\bigcup_{k=0}^{d-1} \mathcal{M}_d^{(k)}$ sits as an algebraic subvariety in $\widetilde{\mathcal{M}}_d\subset \mathcal{T}_d^+$; it is defined as the intersection of $\widetilde{\mathcal{M}}_d$ with the hypersurface $$\mathcal{L}\, :=\, \left\{T\in\mathcal{T}_d^+ \; : \; \det(T)=0\right\}.$$ Since the hypersurface $\mathcal{L}$ defines an irreducible variety, any variety intersects it properly (with the possibility of an empty intersection) or is contained in it. It is easily seen (with the help of the spectral decomposition for instance) that the set $\mathcal{M}^*_d$ is non--empty under~\eqref{inegc0gamma}; in other words, that there are points in $\widetilde{\mathcal{M}}_d$ not contained in $\mathcal{L}$. Thus, the intersection $\widetilde{\mathcal{M}}_d\cap \mathcal{L}$ has codimension at least one in $\widetilde{\mathcal{M}}_d$, which readily implies~\eqref{objectiflebzero} and completes the proof.
\end{proof}

\subsection{Estimation of the Cumulative Distribution Function of the Effective Signal--to--Noise Ratio}\label{estimcumdistrfctSNReff} In view of~\eqref{inegsnreff}, Problem~\ref{pb3} boils down to finding, for a given $s\ge 0$, a lower bound for the event $M_d\left(I_d+\textrm{SNR}\cdot \transp{H}\cdot H \right)\ge 4sd^2$ when $H$ is chosen randomly from the set $\mathcal{H}_{m,n}\left(C_0, \textrm{SNR} \right)$ according to the distribution of the probability measure $\widetilde{\nu}_d$. From the change of variables operated in~\eqref{chgmtvar} and from Lemma~\ref{mesurergplein}, this amounts to bounding from below the quantity
$$\mathfrak{m}_d(\delta)\, :=\, \widetilde{\nu}_d\left(\left\{H\in \mathcal{M}^*_d\; : \; M_d\left(c_0^{-1/d}\left(\gamma I_d + \transp{H}\cdot H\right)\right) > \delta\right\}\right),$$ where we have set 
\begin{equation}\label{reldelatas}
\delta\, :=\, 4d^2s \gamma^d c_0^{-1/d}
\end{equation}
(note that the definitions of $\mathfrak{m}_d(\delta)$ above and of $m_f(\delta)$ in~\eqref{defprobchol} differ inasmuch as the inequalities defining each of these quantities are reversed. The definition of $\mathfrak{m}_d(\delta)$ is here motivated by the statement of Problem~\ref{pb3}).
  Note that when $H\in \mathcal{M}^*_d$, $$c_0^{-1/d}\left(\gamma I_d + \transp{H}\cdot H\right)\in \Sigma_d^{++}.$$

It follows \sloppy immediately from the definition of the the function $M_d$ in~\eqref{defM_d(Q)} that $M_d\left(c_0^{-1/d}\left(\gamma I_d + \transp{H}\cdot H\right)\right)\ge \gamma c_0^{-1/d}$ in such a way that $$\mathfrak{m}_d(\delta)\, = \, 1 \qquad \textrm{ whenever } \qquad \delta\le\frac{\gamma}{c_0^{1/d}}\cdotp$$ In what follows, it will therefore be assumed without loss of generality that 
\begin{equation}\label{conditiondelta}
\delta\, > \, \frac{\gamma}{c_0^{1/d}}\, :=\, \delta_d^*.
\end{equation}

In order to call on Theorem~\ref{thmchol} under this assumption, one needs to push forward the measure $\widetilde{\nu}_d$ from $\mathcal{M}^*_d$ to the space $\Theta_d^{++}$ as defined in~\eqref{deftheta++} via the maps 
\begin{equation}\label{maps}
\mathcal{M}^*_d \xrightarrow{f} \Sigma_d^{++}\xrightarrow{\widetilde{\varphi}_{chol}^{-1}} \Theta_d^{++} 
\end{equation} 
(cf.~\eqref{phitildechol} for the definition of $\widetilde{\varphi}_{chol}$). The main apparent difficulty in doing so is that the Cholesky decomposition of the matrix $\gamma I_d + \transp{H}\cdot H$ cannot be straightforwardly deduced from to the Cholesky form $\transp{H}\cdot H$ when $H\in\mathcal{M}^*_d$. However, explicit expressions can be given from the general Cholesky algorithm which, as mentioned in \S\ref{approachcholesky}, can be implemented in an very efficient way. Thus, given $H= \left(h_{ij} \right)_{1\le i\le j\le d}\in \mathcal{M}^*_d$, if $L=\left(l_{ij}\right)_{1\le i \le j\le d}\in \Theta_d^{++}$ is the Cholesky form of the matrix $c_0^{-1/d}\left(\gamma I_d + \transp{H}\cdot H\right)\in \Sigma_d^{++}$ (that is, if $\transp{L}\cdot L = c_0^{-1/d}\left(\gamma I_d + \transp{H}\cdot H\right)$), one can express recursively the coefficients $h_{ij}$ as functions of $l_{ij}$ (which is what is needed to apply Theorem~\ref{thmchol}) as follows~: for all $1\le i \le d$, 
\begin{align}\label{hii}
h_{ii}\, = \, \sqrt{\sum_{k=1}^i c_0^{1/d}l_{ki}^2-\gamma -\sum_{k=1}^{i-1}h_{ki}^2}
\end{align}
and, for all $1\le i < j \le d$, 
\begin{align}\label{hij}
h_{ij}\, = \, \frac{1}{h_{ii}}\left(\sum_{k=1}^{i}c_0^{1/d}l_{ki}l_{kj} - \sum_{k=1}^{i-1}h_{ki}h_{kj} \right)
\end{align}
(this is just the classical algorithm giving the Cholesky decomposition applied to the positive definite matrix $c_0^{1/d}\cdot\transp{L}\cdot L-\gamma I_d$ --- see~~\cite{watkins} for details).

In order to transport the measure $\widetilde{\nu}_d$ to the space $\Theta_d^{++}$, one will also need to compute the Jacobian $J_{d}$ of the map $f^{-1}\circ\widetilde{\varphi}_{chol}~: \mathcal{N}^*_d\rightarrow \mathcal{M}^*_d$ obtained from~\eqref{maps}, where
\begin{equation}\label{defN*} 
\mathcal{N}^*_d \, :=\, \left(\widetilde{\varphi}_{chol}^{-1}\circ f\right)\left(\mathcal{M}^*_d\right).
\end{equation}
To this end, note that, with the notation of Lemma~\ref{lemjac}, one has $\widetilde{\varphi}_{chol} = \Psi_{(0, 1)}^{(d)}$ and $f= \Psi_{(\gamma, c_0)}^{(d)}$ in such a way that~\eqref{defjaccholmod} implies that 
\begin{equation*}\label{jacdernier}
J_{d} \, = \, c_0^{(d-1)(d+2)/(2d)}\prod_{i=1}^{d-1}\left( \frac{l_{ii}}{h_{ii}}\right)^{d-i+1}.
\end{equation*}

Also, it follows from Lemmata~\ref{lemdetrnirer} and~\ref{mesurergplein} that it is enough to consider the restriction of the measure $\widetilde{\nu}_d$ to the chart $\mathcal{B}_{dd}$ defined from~\eqref{defcartesBij}. It is given therein by the volume element~\eqref{elmtvol} with $i=j=d$.

In view of formulae~\eqref{hii} and~\eqref{hij}, any expression involving the coefficients $h_{ij}$ of a matrix $H\in \mathcal{M}^*_d$ can be viewed as a function of the coefficients $l_{ij}$ of the matrix $L$ as defined above. With this in mind, define two auxiliary functions $\widetilde{J}_{d}$ and $\widetilde{\Gamma}_{d}$ over the space $\mathcal{N}^*_d$ by setting  
\begin{equation}\label{defdensitgamma}
\widetilde{J}_{d}(L)\, :=\, J_{d} \qquad \textrm{ and } \qquad \widetilde{\Gamma}_{d}(L)\, :=\, \left(\frac{\left\|\nabla g \right\|_2}{\left| \partial_{dd}g\right|} \right) (H).
\end{equation} 
Furthermore, if $L\in\Theta_d^{++}$ is decomposed as $L=(\bm{\beta'}, \bm{u})$ with $\bm{\beta'}\in(\R_{>0})^{d-1}$ and $\bm{u}\in\R^p$ as in \S\ref{approachcholesky} (see Equation~\eqref{deftheta++} sqq.~for the notation), it will be convenient to set 
\begin{equation*}
\widetilde{J}_{d}(\bm{\beta'}, \bm{u})\, :=\,\widetilde{J}_{d}(L) \qquad \textrm{ and } \qquad \widetilde{\Gamma}_{d}(\bm{\beta'}, \bm{u})\, :=\, \widetilde{\Gamma}_{d}(L).
\end{equation*}

The main result of this section, which is a direct consequence of the upper bound in~\eqref{cholestim2}, can now be stated as follows~:

\begin{thm}\label{thmprimncisingalpro}
Assume~\eqref{inegc0gamma}, \eqref{conditiondelta} and also that $\delta<1$. Then, 
\begin{equation}\label{formulethm5}
\mathfrak{m}_d(\delta)\,\ge\, \kappa_d^{-1}\cdot \int_{\mathcal{N}^*_d\left[\delta\right]} \widetilde{J}_{d}(\bm{\beta'}, \bm{u}) \cdot \widetilde{\Gamma}_{d}(\bm{\beta'}, \bm{u})\cdot  \emph{\textrm{d}}\lambda_{p+d-1}\left(\bm{\beta'}, \bm{u} \right).
\end{equation}

Here, $$\mathcal{N}^*_d\left[\delta\right]\, :=\, \left\{(\bm{\beta'}, \bm{u})\in \mathcal{N}^*_d\; : \; \bm{\beta'}\in \Delta_{d-1}(\delta)\right\}$$ is a subset of $\mathcal{N}^*_d$ , $\Delta_{d-1}(\delta)$ is defined as in~\eqref{cholestim2} and $$\kappa_d\, :=\, \int_{\widetilde{\mathcal{M}}_d}\emph{\textrm{d}}\, \emph{\textrm{vol}}_{p'}(H)$$ is the area of the hypersurface $\widetilde{\mathcal{M}}_d$.
\end{thm}

In view of Lemmata~\ref{lemdetrnirer} and~\ref{mesurergplein}, the constant $\kappa_d$ can also be computed with the help of any of the following formulae~: 
\begin{align}\label{reformkappa}
\kappa_d\, &=\, \int_{\mathcal{M}^*_d} \left(\frac{\left\|\nabla g \right\|_2}{\left| \partial_{dd}g\right|} \right) (H)\cdot \textrm{d}h_{11}\dots\textrm{d}h_{d,d-1}\\
&=\, \int_{\mathcal{N}^*_d} \widetilde{J}_{d}(\bm{\beta'}, \bm{u}) \cdot \widetilde{\Gamma}_{d}(\bm{\beta'}, \bm{u})\cdot  \emph{\textrm{d}}\lambda_{p+d-1}\left(\bm{\beta'}, \bm{u} \right).\label{reformkappabis}
\end{align} 

A direct use of~\eqref{reformkappa} requires that the coefficient $h_{dd}$ be expressed as a function of the other entries of the matrix $H$. To this end, it should be mentioned that, as established in the course of the proof of Lemma~\ref{lemdetrnirer} below, the coefficient $h_{dd}$ appears only once (in the form $h_{dd}^2$) in the determinant defining the set $\widetilde{\mathcal{M}}_d$ in~\eqref{defmtildedgamc} --- see \S\ref{pruevelem6} for details. 

If one wants cruder but simpler--to--obtain estimates for the right--hand side of~\eqref{formulethm5}, it should first be noted that the density function $\widetilde{\Gamma}_{d}$ defined in~\eqref{defdensitgamma} and appearing in~\eqref{formulethm5} and~\eqref{reformkappabis} as a function of $L$ and in~\eqref{reformkappa} as a function of $H$ is clearly bounded below by 1. In order to bound it from above, one can bound the gradient therein from above with the help of Remark~\ref{rem2}. Also, the explicit formula given in Equation~\eqref{partielddg} below for the partial derivative $\left(\partial_{dd}g\right)(H)$ can easily be used to bound the latter quantity from below as a function of $h_{dd}$, $\gamma$ and $c_0$.

The lower bound appearing in Theorem~\ref{thmprimncisingalpro} involves the computation of the integral of an algebraic function (more precisely~: the square root of some rational function) over an algebraic domain (which can be explicitly defined with the help of inequalities involving polynomials). This can certainly be done numerically in such a way that Theorem~\ref{thmprimncisingalpro} can be seen as a way to obtain numerical values for the quantity $\mathfrak{m}_d(\delta)$. A more theoretical approach would necessarily require involved calculations which can nevertheless be carried out for a fixed value of $d$. 

As mentioned in \S\ref{positionpb}, the case of $d=m=2$ users and $n=2$ receivers is already of interest in the theory of Signal Processing. We explicitly work out the estimates that can be obtained from Theorem~\ref{thmprimncisingalpro} in this case. In order to put the emphasis on the behaviour of the probability $\mathfrak{m}_2(\delta)$ as a function of $\delta$ and in order not to introduce unnecessary cumbersome definitions, we present the result in the following way, where an explicit expression for the function $\chi$ follows immediately from the proof presented in \S\ref{preuvecoro} (see Equation~\eqref{defchi} below)~:

\begin{coro}\label{coroprimncisingalpro}
Assume that $c_0>\gamma^2$ and that $\delta_2^* :=\gamma/c_0^{1/2} < \delta < 1$. Then,  there exists a function $\chi$ such that 
\begin{equation}\label{estimm2}
\mathfrak{m}_2(\delta)\,\ge\, \gamma^{-1} c_0^{-1/2}\cdot\int_{\sqrt{\delta}}^{1/\sqrt{\delta}}\frac{\emph{\textrm{d}}a}{\sqrt{c_0^{1/2}a^2-\gamma}} \int_{-\theta (a)}^{\theta (a)}\emph{\textrm{d}}b\cdot \frac{\chi (a,b)}{\sqrt{\theta (a)^2-b^2}}\, :=\, \mathfrak{n}_2(\delta),
\end{equation}
where 
\begin{equation}\label{deftheta}
\theta (a)\,:=\,\sqrt{\frac{1}{\gamma c_0^{1/2}}\cdot \left(\frac{c_0^{1/2}}{a^2}-\gamma\right)\cdot\left(c_0^{1/2}a^2-\gamma\right)}
\end{equation} 
and where the right--hand side is equal to 1 when $\delta=\delta_2^*$. 

Furthermore, the function $\chi$ takes its values in a interval of the form $[\omega_1, \, \omega_2]$, where the constants $\omega_1$ and $\omega_2$ are such that $0<\omega_1<\omega_2<+\infty$ and depend only on $\gamma$ and $c_0$.
\end{coro}

The corollary implies that the probability $\mathfrak{m}_2(\delta)$ tends to 1 as $\delta$ tends to the critical value $\delta_2^*$ with an error term governed by the size of the difference $\mathfrak{n}_2(\delta_2^*) - \mathfrak{n}_2(\delta)$. Note that upon bounding the function $\chi$ from above by the constant $\omega_2$, the inner integral in~\eqref{estimm2} becomes independent of the variable $a$. This shows that the error term in the difference $1-\mathfrak{m}_2(\delta)$ is, up to a multiplicative constant, bounded above by $$\left(\int_{\sqrt{\delta_2^*}}^{1/\sqrt{\delta_2^*}}-\int_{\sqrt{\delta}}^{1/\sqrt{\delta}} \right)\frac{\textrm{d}a}{\sqrt{c_0^{1/2}a^2-\gamma}}\; = \; O\left(\delta-\delta_2^*\right) $$ (this relation follows from a direct evaluation of the integral in the left--hand side. Details of the calculations are left as an exercise for the interested reader). We thus recover when $d=2$ the growth in $\delta^{d/2}$ as in Theorem~\ref{thmkleimarg}.

Typical values for the capacity $C_0$ of a channel and for the Signal--to--Noise Ratio $SNR$ can be taken as $C_0=30$ bits and $SNR=5$ dB. From the expression for the function $\chi$ deduced from the proof of Corollary~\ref{coroprimncisingalpro}, one can find an explicit lower bound for the probability that the Effective Signal--to--Noise Ratio $SNR_{eff}$ should be bigger than a given value $s\ge 0$. From the discussion held at the beginning of \S\ref{estimcumdistrfctSNReff}, this amounts to bounding from below the quantity $\mathfrak{m}_2(\delta)$ when $\delta$ (hereafter denoted by $\delta_s$) is viewed as a function of $s$ according to~\eqref{reldelatas}. Note that with such choices, $\gamma = 1/5$ and $c_0=e^{30}/25$. Furthermore, $\delta_2^* =e^{-15}\approx  3.06\cdot 10^{-7}$ arises from the limit value $s_2^*=5/16=0.3125$. Some numerical values are recorded in the following table.\\

\begin{center}
\begin{tabular}{||c||c|c|c|c|c||}
\hline
\hline
$s$ & $s_2^*=0.3125$ & $1$ & $1.5$ & $2$  \\
\hline
$\delta_s \approx $ & $3.06\cdot 10^{-7}$ & $9.79\cdot 10^{-7}$ & $1.47\cdot 10^{-6}\cdot 10^{-7}$ & $1.96\cdot 10^{-6}\cdot 10^{-7}$ \\
\hline
$\mathfrak{m}_2(\delta_s)\geq$ & 1 & $0.672723$ & $0.560289$ & $0.489859$ \\
\hline
\hline
\end{tabular}

\begin{tabular}{||c||c|c|c||}
\hline
\hline
$s$ & $5$ & $10$ & $30$\\
\hline
$\delta_s$ & $4.90\cdot 10^{-6}\cdot 10^{-7}$ & $9.79\cdot 10^{-6}\cdot 10^{-7}$ & $2.94\cdot 10^{-5}$  \\
\hline
$\mathfrak{m}_2(\delta_s)\geq$ & $0.314961$ & $0.223899$  & $0.12972$  \\
\hline
\hline
\end{tabular}
\end{center}

Thus, for instance, to ensure that the event $SNR_{eff}\ge s$ occurs with probability at least 45\%, it is enough to choose $s=2$. Also, the initial value of $SNR=5$ is recovered  with probability at least 31\%.

As a concluding remark, we would like to mention here that, from a numerical point of view, the computation of the Cholesky transforms required to estimate the integrals in Theorem~\ref{thmprimncisingalpro} can be implemented in a much more efficient and stable way than using Equations~\eqref{hii} and~\eqref{hij}. For further details, the interested reader is referred to~\cite{watkins} and to the references therein. 

The rest of this section is devoted to the proofs of Lemma~\ref{lemdetrnirer} and Corollary~\ref{coroprimncisingalpro}.

\subsection{Proof of Lemma~\ref{lemdetrnirer}}\label{pruevelem6}
The second point is proved in~\cite[Chap.~11, \S C]{jones}. 

As for the first point, given $T:=\left(t_{ij}\right)_{1\le i \le j \le d}\in \widetilde{\mathcal{M}}_d$ and $\beta>0$, consider the homogeneous polynomial $F$ of degree $2d$ defined as $$F(T, \beta)\,:=\, \det\left(\beta^2I_d+\transp{T}\cdot T \right).$$ Note that 
\begin{equation}\label{lienFg}
F(T, \gamma^{1/2})\, \underset{\eqref{defg}}{=}\, c_0\cdot g(T)
\end{equation} 
and assume for a contradiction that 
\begin{equation}\label{conditionderivparnulle}
\partial_{ij} F(T, \gamma^{1/2})=0
\end{equation} 
for all $1\le i\le j\le d$.

It follows from Euler's formula for the derivative of a homogeneous function that $$2d\cdot F(T, \beta)\, = \, \sum_{1\le i \le j \le d}t_{ij}\cdot \partial_{ij} F(T, \beta) + \beta\cdot \partial_\beta F(T, \beta)$$ (here, $\partial_\beta$ obviously denotes the partial derivative with respect to the last variable $\beta$). Under~\eqref{conditionderivparnulle}, this implies that 
\begin{equation}\label{provisoireF}
2d\cdot F\left(T, \gamma^{1/2}\right)\, = \, \gamma^{1/2}\cdot \partial_\beta F(T, \gamma^{1/2}).
\end{equation}

Let $\llbracket 1, d\rrbracket$ denote the interval of positive integers less than $d$. Given $K\subset\llbracket 1, d\rrbracket$, denote furthermore by $\left|K \right|$ the cardinality of $K$ and by $m_K$ the $\left|K \right|\times \left|K \right|$ matrix obtained by considering the rows and columns indexed by $K$ in the matrix $\transp{T}\cdot T$. Set conventionally $$\det m_{\emptyset}\,:=\,1.$$ As $m_K$ is the Gramian matrix of the columns of $T$ indexed by $K$, $\det m_K$ is non--negative. Furthermore, the definition of the determinant readily implies that 
\begin{equation}\label{decompodet}
F(T, \beta)\, = \, \sum_{K\subset\llbracket 1, d\rrbracket} \beta^{2d-2\left| K\right|} \det m_K.
\end{equation}
Differentiating with respect to $\beta$ and multiplying throughout by $\beta$ then yields
\begin{equation}\label{derivdecompodet}
\beta\cdot\partial_{\beta} F(T, \beta)\, = \, \sum_{K\subset\llbracket 1, d\rrbracket} (2d-2\left|K\right|) \beta^{2d-2\left| K\right|} \det m_K.
\end{equation}
On combining~\eqref{provisoireF}, \eqref{decompodet} and~\eqref{derivdecompodet}, one thus obtains the relation $$2d \sum_{K\subset\llbracket 1, d\rrbracket} \gamma^{d-\left| K\right|} \det m_K \, = \,  \sum_{K\subset\llbracket 1, d\rrbracket} (2d-2\left|K\right|) \gamma^{d-\left| K\right|} \det m_K,$$ i.e. 
\begin{equation*}
\sum_{K\subset\llbracket 1, d\rrbracket} 2\left|K\right| \gamma^{d-\left| K\right|} \det m_K\, = \, 0.
\end{equation*}
Since each term on the left--hand side of this equation is positive, this implies that $\det m_K = 0$ for all non--empty $K\subset \llbracket 1, d\rrbracket$, i.e.~$T=\bm{0}$. Under assumption~\eqref{inegc0gamma}, this contradicts the fact that $T\in \widetilde{\mathcal{M}}_d$ and thus concludes the proof of the first point.

The third point is elementary~: given $T\in \mathcal{M}^*_d$, the coefficient $t_{dd}$ appears only in the bottom right corner in the matrix $\gamma I_d + \transp{T}\cdot T$, where it is present as $t_{dd}^2$. Thus, after expanding the determinant $g(T)$ following the last column, one obtains that 

\begin{equation}\label{partielddg}
\left(\partial_{dd} g\right)(T)\, = \, c_0^{-1}\cdot 2t_{dd} \cdot \det\left(\gamma I_{d-1}+\transp{T'}\cdot T' \right),
\end{equation} 
where the matrix $T'$ is obtained by stripping off the matrix $T$ from its last column and row. Clearly, the latter quantity does not vanish under the assumption that $T$ has full rank. This concludes the proof of the lemma.

\paragraph{} The claims made in Remark~\ref{rem2} can now be justified as follows~: given $T\in\widetilde{\mathcal{M}}_d$ denote by $\bm{t_i}$ ($1\le i\le d$) the $i^{\textrm{th}}$ column of the matrix $T$ and by $\bm{t}$ this matrix viewed as a vector in $\R^{d(d+1)/2}$. Upon isolating the terms corresponding to $K=\emptyset$ and $K=\left\{i\right\}$ ($1\le i\le d$) from the others in~\eqref{decompodet}, this equation together with~\eqref{lienFg} readily implies that $\left\|\bm{t} \right\|^2_2\le (c_0-\gamma^d)/\gamma^{d-1}$. Conversely, it follows from Hadamard's inequality that the determinant of the positive definite matrix $\gamma I_d +\transp{T}\cdot T$ is less than or equal to the product of its diagonal entries. Thus, $$c_0\, = \,\det\left(\gamma I_d +\transp{T}\cdot T \right)\, \le \, \prod_{i=1}^{d}\left(\gamma + \left\|\bm{t_i} \right\|^2_2 \right)\, \le \, \left(\gamma + \left\|\bm{t} \right\|^2_2 \right)^d,$$ hence the fact that $\left\|\bm{t} \right\|^2_2\ge c_0^{1/d}-\gamma$.

\subsection{Proof of Corollary~\ref{coroprimncisingalpro}}\label{preuvecoro} Let $$H\,:=\, \begin{pmatrix}
u& v\\
0& w
\end{pmatrix}\, \in\, \mathcal{M}^*_2$$ and $$L\,:=\, \begin{pmatrix}
a& b\\
0& 1/a
\end{pmatrix} \,\in\, \Theta^{++}_2$$ be such that $$\transp{L}\cdot L \, = \, c_0^{-1/2}\left( \gamma I_2 +\transp{H}\cdot H\right).$$ Formulae~\eqref{hii} and~\eqref{hij} then read $$u\, = \, \sqrt{c_0^{1/2}a^2-\gamma}, \qquad v\, = \, \frac{c_0^{1/2}ab}{\sqrt{c_0^{1/2}a^2-\gamma}}$$ and $$w\, = \, \sqrt{c_0^{1/2}b^2+\frac{c_0^{1/2}}{a^2}-\frac{c_0a^2b^2}{c_0^{1/2}a^2-\gamma}-\gamma}\, = \, \sqrt{\frac{\left(\frac{c_0^{1/2}}{a^2}-\gamma\right)\cdot\left(c_0^{1/2}a^2-\gamma\right)-\gamma c_0^{1/2}b^2}{c_0^{1/2}a^2-\gamma}}\cdotp$$
This is easily seen to imply that the set $\mathcal{N}^*_2$ defined in~\eqref{defN*} can be explicitly expressed as follows~: 
$$\mathcal{N}^*_2\, = \, \left\{(a, b)\in \R_{>0}\times \R\; : \; \sqrt{\delta_2^*}< a< (\sqrt{\delta_2^*})^{-1}\quad \textrm{ and }\quad \left| b\right|< \theta (a) \right\},$$ where the quantity $\theta (a) $ has been defined in~\eqref{deftheta}.

Furthermore, the function $g$ defined in~\eqref{defg} reads in this case $$g(u, v, w)\, = \, c_0^{-1}\cdot \left(\left( u^2+\gamma\right)\cdot \left(w^2+\gamma \right) +\gamma v^2\right)$$ and, with the notation of Theorem~\ref{thmprimncisingalpro}, $$\widetilde{J}_{2}(a, b) \cdot \widetilde{\Gamma}_{2}(a, b)\, = \, \left(c_0\cdotp \frac{a^2}{u^2(a,b)} \right)\cdot \left(\frac{\widetilde{g}(a, b)}{2c_0^{-1}\cdot w(a, b)\cdot (u^2(a,b)+\gamma)}\right).$$ In this equation, the variables $u$ and $w$ are seen as functions of $a$ and $b$ and $\widetilde{g}$ is the norm of the gradient of $g$ (with respect to $u, v$ and $w$) also expressed as a function of the parameters $a$ and $b$; that is, with obvious notation, $$\widetilde{g}(a,b)\, :=\, \left(\left\|\nabla_{(u,v,w)}\, g\right\|_2 \right)(a,b).$$

Set 
\begin{equation}\label{defchi}
\chi(a, b)\, :=\, \frac{c_0^2}{2\kappa_2}\cdot \frac{a^2\cdot \widetilde{g}(a, b)}{u^2(a,b)+\gamma},
\end{equation} 
where $\kappa_2$ is the constant defined for instance in~\eqref{reformkappabis}.

The existence of the constants $\omega_1$ and $\omega_2$ is then guaranteed by the fact the parameter $a$ stays bounded away from zero (see the expression of $u$ above) and the fact that the gradient of $g$ is continuous and never vanishes on the compact set $\widetilde{\mathcal{M}}_d$ (see Lemma~\ref{lemoutilpropH} and Remark~\ref{rem2}).

Note also that $$u^2(a,b)\cdot w(a,b)\; =\; \gamma c_0^{1/2}\cdot \sqrt{c_0^{1/2}a^2-\gamma} \cdot \sqrt{\theta^2(a)-b^2}.$$ 

In order to conclude the proof, one needs to show that the right--hand side of~\eqref{estimm2} is equal to 1 when $\delta=\delta_2^*$; that is, that $\mathfrak{n}_2(\delta_2^*)=1$. With the notation of Theorem~\ref{thmprimncisingalpro}, this readily follows from the fact that $$\mathcal{N}^*_2\left[\delta_2^*\right]\, = \, \mathcal{N}^*_2$$ (such a relation does not hold any more in dimension $d\ge 3$).\\

\paragraph{\textbf{Acknowledgement}}
The main catalyst for this work was the International Workshop on Interactions between Number Theory and Wireless Communication held at the University of York between 9--23 May 2014. The authors would like to thank the engineers, especially Uri Erez, Bobak Nazer and Or Ordentlich, for providing them with such an interesting topic of research which has turned out to be related to deep theoretical questions. The authors hope that this work will contribute to foster further collaboration between Number Theorists and Engineers.

\end{document}